\documentclass[
final
]{dmtcs-episciences}

\usepackage[utf8]{inputenc}
\usepackage{subfigure}
\usepackage{hyperref}
\usepackage{mathtools}
\usepackage{amsmath,amssymb,amsfonts,amsthm,tikz,enumerate}
\usepackage{xcolor}
\usepackage[T1]{fontenc}
\usepackage{array}
\usepackage{multirow}

\tikzstyle{pnt}=[draw,circle,fill,inner sep=1pt]

\newcommand\arc[2]{ \draw[blue,very thick] (#1)  to [bend left=45] (#2);}
\newcommand\arcred[2]{ \draw[red,very thick] (#1)  to [bend left=45] (#2);}

\newcommand{\cS}{\mathcal{S}}
\newcommand{\cC}{\mathcal{C}}
\newcommand{\cD}{\mathcal{D}}
\newcommand\nn{[n]\sqcup[n]}
\newcommand\cc{c}
\newcommand\cs{s}
\newcommand\emptyw{\varepsilon}
\newcommand\card[1]{\left|#1\right|}
\DeclareMathOperator{\Set}{Set}
\newcommand{\bbN}{\mathbb{N}}
\DeclareMathOperator{\st}{st}
\newcommand{\Cat}{C}
\newcommand{\Fib}{F}
\newcommand\hpi{\hat\pi}
\newcommand\one{a^{(1)}}
\newcommand\two{a^{(2)}}
\newcommand\three{a^{(3)}}
\newcommand\four{a^{(4)}}
\newcommand\colI{\textcolor{red}}
\newcommand\colII{\textcolor{orange}}
\newcommand\colU{\textcolor{blue}}
\newcommand\colD{\textcolor{violet}}
\newcommand\uu{\mathtt{u}}
\newcommand\dd{\mathtt{d}}
\DeclareMathOperator{\des}{des}

\newtheorem{theorem}{Theorem}[section]

\newtheorem{lemma}[theorem]{Lemma}
\newtheorem{corollary}[theorem]{Corollary}
  
\newtheorem{problem}{Problem}

\theoremstyle{definition}

\newtheorem{remark}[theorem]{Remark}
\usepackage[round]{natbib}

\author{Sergi Elizalde\affiliationmark{1}\thanks{Partially supported by Simons Collaboration Grant \#929653}
  \and Amya Luo\affiliationmark{2}\thanks{Partially supported by a Presidential Scholarship from Dartmouth College}}
\title{Pattern avoidance in nonnesting permutations}
\affiliation{
  Department of Mathematics, Dartmouth College, Hanover, NH, USA\\
  School of Mathematics, Georgia Institute of technology, Atlanta, GA, USA}
\begin{document}
\publicationdata
{vol. 27:1 Permutation Patterns 2024}
{2025}
{13}
{10.46298/dmtcs.14885}
{2025-05-30; 2025-09-03}
{2025-09-04}
\maketitle
\begin{abstract}
Nonnesting permutations are permutations of the multiset $\{1,1,2,2,\dots,n,n\}$ that avoid subsequences of the form $abba$ for any $a\neq b$. 
These permutations have recently been studied in connection to noncrossing (also called quasi-Stirling) permutations, which are those that avoid subsequences of the form $abab$, and in turn generalize the well-known Stirling permutations. 
Inspired by the work by Archer et al.\ on pattern avoidance in noncrossing permutations, we consider the analogous problem in the nonnesting case. We enumerate nonnesting permutations that avoid each set of two or more patterns of length~3, as well as those that avoid some sets of patterns of length~4.
We obtain closed formulas and generating functions, some of which involve unexpected appearances of the Catalan and Fibonacci numbers.
Our proofs rely on decompositions, recurrences, and bijections. 
\end{abstract}

\section{Introduction}\label{sec1}

Let $[n]=\{1,2,\dots,n\}$, and let $\cS_n$ be the set of permutations of $[n]$.
We denote by $\nn = \{1,1,2,2,\dots, n,n\}$ the multiset consisting of two copies of each integer between $1$ and $n$.

Given two words $\pi=\pi_1\pi_2\dots\pi_m$ and $\sigma=\sigma_1\sigma_2\dots\sigma_k$ over the positive integers $\bbN$, we say that $\pi$ {\em contains} the pattern $\sigma$ if there exist indices $1\le i_1<i_2<\dots<i_k\le m$ such that the subsequence $\pi_{i_1}\pi_{i_2}\dots\pi_{i_k}$ is in the same relative order as $\sigma$, that is, 
\begin{itemize}
\item $\pi_{i_r}<\pi_{i_s}$ if and only if $\sigma_r<\sigma_s$, and
\item $\pi_{i_r}=\pi_{i_s}$ if and only if $\sigma_r=\sigma_s$,
\end{itemize}
for all $r,s\in[k]$. This subsequence is called an {\em occurrence} of $\sigma$. If $\pi$ does not contain $\sigma$, we say that $\pi$ {\em avoids} the pattern $\sigma$.

A {\em Stirling permutation} is a permutation $\pi$ of $\nn$ that avoids the pattern $212$; equivalently, there do not exist indices $1\le i_1<i_2<i_3\le 2n$ such that $\pi_{i_2}<\pi_{i_1}=\pi_{i_3}$. Stirling permutations were introduced by \cite{gessel_stanley} in connection to certain generating functions for Stirling numbers of the second kind.

A {\em quasi-Stirling permutation} is a permutation $\pi$ of $\nn$ that avoids the patterns $1212$ and $2121$; equivalently, there do not exist indices $1\le i_1<i_2<i_3<i_4\le 2n$ such that $\pi_{i_1}=\pi_{i_3}$ and $\pi_{i_2}=\pi_{i_4}$.
Quasi-Stirling permutations were introduced by \cite{archer} as a generalization of Stirling permutations that arises in connection with labeled trees, and were further studied by~\cite{elizalde_quasiStirling}.

One can view permutations $\pi$ of $\nn$ as labeled matchings of $[2n]$, by placing an arc with label $\ell$ between $i$ and $j$ if $\pi_i=\pi_j=\ell$. With this interpretation, a permutation of $\nn$ is quasi-Stirling if and only if the corresponding matching is noncrossing, i.e., there are no two arcs $(i_1,i_3)$ and $(i_2,i_4)$ where $i_1<i_2<i_3<i_4$. For this reason, quasi-Stirling permutations are also called {\em noncrossing permutations} in~\cite{elizalde_nonnesting}.

With this perspective, it is natural to consider permutations of $\nn$ whose corresponding matching is nonnesting, i.e., there are no two arcs $(i_1,i_4)$ and $(i_2,i_3)$ where $i_1<i_2<i_3<i_4$. They can be defined as permutations of $\nn$ that avoid the patterns $1221$ and $2112$. Following~\cite{elizalde_nonnesting}, we call these {\em nonnesting permutations}, and we denote by $\cC_n$ the set of nonnesting permutations of $\nn$.
For example, as shown in Figure~\ref{fig:arc_diagram}, $1521352434\in\cC_5$, but $13241342\notin\cC_4$, because the subsequence $2442$ is in the same relative order as $1221$.
It is well known (see \cite{Stanley_Catalan}) that both noncrossing and nonnesting matchings of $[2n]$ are counted by the $n$th Catalan number $\Cat_n=\frac{1}{n+1}\binom{2n}{n}$. Since there are $n!$ ways to assign the labels to the arcs, it follows (see~\cite{elizalde_nonnesting}) that both the number noncrossing and the number of nonnesting permutations of $\nn$ are given by
$$|\cC_n| = n!\Cat_n=\frac{(2n)!}{(n+1)!}.$$

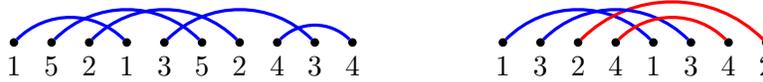
\begin{figure}[ht]
    \centering
 \begin{tikzpicture}[scale=0.5]
   \foreach [count=\i] \j in {1,5,2,1,3,5,2,4,3,4}
        \node[pnt,label=below:$\j$] at (\i,0)(\i) {};
   \arc{1}{4} \arc{3}{7} \arc{5}{9} \arc{8}{10} \arc{2}{6}
   
  \begin{scope}[shift={(13,0)}]
   \foreach [count=\i] \j in {1,3,2,4,1,3,4,2}
        \node[pnt,label=below:$\j$] at (\i,0)(\i) {};
   \arc{1}{5} \arcred{3}{8} \arc{2}{6} \arcred{4}{7} 
   \end{scope}
   \end{tikzpicture}
\caption{The permutation $1521352434$ is nonnesting, but the permutation $13241342$ is not.}
\label{fig:arc_diagram}
\end{figure}

In~\cite{archer}, the authors consider quasi-Stirling permutations that avoid other patterns. Specifically, they enumerate quasi-Stirling (i.e.\ noncrossing) permutations that avoid any set of at least two elements from $\cS_3$.
The goal of this paper is to extend the results from~\cite{archer} to the nonnesting case, by providing the enumeration of nonnesting permutations that avoid any set of at least two elements from $\cS_3$, as well as those that avoid some patterns of length~4.

The known bijections between noncrossing and nonnesting matchings (see e.g.~\cite{athanasiadis}), when extended to permutations of $\nn$, do not generally behave well with respect to pattern avoidance. In particular, with the exception described in Remark~\ref{rem:nonnesting}, there is no straightforward way to translate the results from~\cite{archer} to our setting. This is reflected in the fact that the enumeration formulas that we obtain in the nonnesting case are mostly different from those in \cite[Fig.~6]{archer}.

Noncrossing and nonnesting permutations also have a very different behavior when enumerated with respect to the number of descents.
For noncrossing permutations, this refined enumeration is obtained in~\cite{elizalde_quasiStirling} by using a recursive decomposition of certain rooted trees to derive an implicit formula for the corresponding bivariate generating function. On the other hand, it is shown in~\cite{elizalde_nonnesting} that the distribution of the number of descents on nonnesting permutations has some unexpected properties, such as being symmetric, as well as close connections to standard Young tableaux (see~\cite{Elizalde_canon}). As we will discuss in Section~\ref{sec:further}, analogous properties hold for nonnesting permutations avoiding certain patterns.

Additional motivation for the study of pattern-avoiding nonnesting permutations comes from Bernardi's work on deformations of the braid arrangement (see~\cite{bernardi}). He shows that certain configurations, called {\em annotated $1$-sketches}, naturally index the regions of the Catalan arrangement, and that certain subsets of them index the regions of other important hyperplane arrangements. It turns out that annotated $1$-sketches are precisely nonnesting permutations, and that the relevant subsets can be described as nonnesting permutations avoiding {\em vincular patterns}, which are patterns where some entries are required to be adjacent in an occurrence. For example, regions of the semiorder arrangement are indexed by permutations $\pi\in\cC_n$ with no subsequence $\pi_i\pi_{i+1}\pi_j\pi_{j+1}$ such that $\pi_i=\pi_j<\pi_{i+1}=\pi_{j+1}$ (in vincular pattern notation, we say that $\pi$ avoids $12\text{-}12$), regions of the Shi arrangement are indexed by permutations $\pi\in\cC_n$ with no subsequence $\pi_i\pi_{j}\pi_{j+1}\pi_{k}$ such that $\pi_i=\pi_j<\pi_{j+1}=\pi_k$ (we say that $\pi$ avoids $1\text{-}12\text{-}2$), and regions of the Linial arrangement are indexed by permutations $\pi\in\cC_n$ that avoid both $12\text{-}12$ and $1\text{-}12\text{-}2$. The numbers of regions of such hyperplane arrangements are well known, so one immediately deduces formulas for the number of nonnesting permutations avoiding these specific vincular patterns. This suggests the problem of enumerating nonnesting permutations that avoid other patterns. 
In this paper we will tackle this problem for the case of classical patterns, that is, with no adjacency requirements. In general, the number of permutations avoiding a vincular pattern is bounded from below by the number of permutations avoiding the corresponding classical pattern where the adjacency requirements have been removed.

In Section~\ref{sec:length3}, we study nonnesting permutations that avoid sets of patterns of length~$3$, completing the enumeration for all sets $\Lambda\subseteq\cS_3$ consisting of at least $2$ patterns, in analogy with the work in~\cite{archer} for the noncrossing case.
In Section~\ref{sec:length4}, we enumerate nonnesting permutations avoiding several sets of patterns of length~$4$, which often require more complicated proofs. Our results include some unexpected new interpretations of the Catalan numbers. We conclude with a few conjectures in Section~\ref{sec:further}.

The proof techniques include bijections, generating functions (both ordinary and exponential), and decompositions of the permutations into smaller pieces obtained by analyzing their structure, which often give rise to recurrences or summation formulas. 

Let us finish this section by introducing some notation.
Given a set $\Lambda$ of finite words over $\bbN$, let $\cC_n(\Lambda)$ denote the set of permutations in $\cC_n$ that avoid all the patterns in $\Lambda$, and let its cardinality be $\cc_n(\Lambda)=\card{\cC_n(\Lambda)}$. If $\Lambda=\{\sigma,\tau,\dots\}$, we often write $\cC_n(\sigma,\tau,\dots)$ instead of $\cC_n(\{\sigma,\tau,\dots\})$.

For any word $\alpha = \alpha_1 \alpha_2 \dots \alpha_k$ over $\bbN$, define its reversal by $\alpha^r=\alpha_k\dots\alpha_2\alpha_1$. If $n$ is the largest entry of $\alpha$, define its complement $\alpha^c$ to be the word whose $i$th entry is $n+1-\alpha_i$ for all $i\in [k]$. The composition of these two operations gives the reverse-complement $\alpha^{rc}$.
Clearly, $\pi$ avoiding $\sigma$ is equivalent to $\pi^r$ avoiding $\sigma^r$, to $\pi^c$ avoiding $\sigma^c$, and to $\pi^{rc}$ avoiding $\sigma^{rc}$. In particular, the set $\cC_n$ is closed under the operations of reversal and complementation.

Denote by $\st(\alpha)$ the {\em standardization} of $\alpha$, which is obtained by replacing the copies of the smallest entry with $1$, the copies of the second smallest entry with $2$, and so on.

Denote by $\Set(\alpha)$ the set of different entries in $\alpha$,  without multiplicities. For example, we have $\Set(113232) = \{1,2,3\}$. Given two sets $A$ and $B$, we write $A<B$ to mean that $a < b$ for every $a \in A$ and $b \in B$. Given two words 
$\alpha$ and $\beta$, we write $\alpha<\beta$ to mean $\Set(\alpha)<\Set(\beta)$. We define other relations $>$, $\leq$ and $\geq$ similarly. Note that $A \leq B$ implies that $\card{A\cap B}\leq1$. 
Note also that these relations are not transitive or antisymmetric, since the empty word, which we denote by $\emptyw$, trivially satisfies that $\emptyw\le\alpha$ and $\emptyw\ge\alpha$ for any $\alpha$.

If $\pi\in\cC_n$ avoids a pattern $\sigma$, then the permutation $\pi'\in\cC_{n-1}$ obtained by removing the two copies of $n$ from $\pi$ also avoids $\sigma$. In this case, we will say that $\pi'$ {\em generates} $\pi$.


\section{Patterns of length~3}\label{sec:length3}

In this section we consider nonnesting permutations avoiding patterns of length~3. Applying reversal and complementation, the enumeration of $\cC_n(\Lambda)$ for all $\Lambda\subseteq\cS_3$ with $|\Lambda|\ge 2$ can be reduced to the sets listed in Table~\ref{tab:length3}, which serves as a summary of the results in this section.

\begin{table}[htb]
\centering
\begin{tabular}{|c|c|c|c|}
        \hline
        $\Lambda$ & Formula for $\cc_n(\Lambda)$ & OEIS code & Result in the paper \\
                \hline\hline
        $\{112\}$ & $\Cat_n$ & A000108 & Theorem~\ref{thm:112} \\
        \hline
        $\{121\}$ & $n!$ & A000142 & Theorem~\ref{thm:121} \\
        \hline
        $\{123,321\}$    & {$0$, for $n \geq 5$}  & N/A & Corollary~\ref{cor:123,321} \\
        \hline
        $\{123,231\}$    & $\dfrac{n^2+5n-6}{2} $, for $n \geq 2$  & A055999 & Theorem~\ref{thm:123,231}  \\
        \hline
        $\{132,213\}$    & $\Fib_n^2$  & A007598  & Theorem~\ref{thm:132,213}\\
        \hline
        $\{132,231\}$    &    $2^n$, for $n \geq 2$ & A000079 &     Theorem~\ref{thm:132,231} \\
        \hline
        $\{132,312\}$     & \multirow{2}{*}{$4\cdot3^{n-2}$, for $n \geq 2$} & \multirow{2}{*}{A003946}   & Theorem~\ref{thm:132,312} \\
        \cline{1-1}
        \cline{4-4}
        $\{123,213\}$     & &  & Theorem~\ref{thm:123,213}\\
        \hline
        $\{123,132,213\}$     & {OGF: $\dfrac{1-x}{1-2x-2x^2+2x^3}$} & A052528 & Theorem~\ref{thm:123,132,213} \\
        \hline
        $\{123,213,312\}$ & \multirow{2}{*}{$n+2$ , for $n \geq 2$} & \multirow{2}{*}{A000027} & Theorem~\ref{thm:123,213,312} \\
        \cline{1-1}
        \cline{4-4}
        $\{132,213,312\}$ & & & Theorem~\ref{thm:132,213,312}\\
        \hline
        $\{123,231,312\}$ & $n$, for $n \geq 3 $ & A000027 & Theorem~\ref{thm:123,231,312} \\
        \hline
        $\{123,213,231\}$ & $4(n-1)$, for $n \geq 2$  & A008586 & Theorem~\ref{thm:123,213,231}\\
        \hline
        $\{123,132,213,231\}$ & $4$, for $n \geq 2$ & N/A & Theorem~\ref{thm:123,132,213,231} \\
        \hline
        $\{123,132,231,312\}$ & \multirow{2}{*}{$2$, for $n \geq 3$} & \multirow{2}{*}{N/A} & Theorem~\ref{thm:123,132,231,312} \\
        \cline{1-1}
        \cline{4-4}
        $\{132,213,231,312\}$ & & & Theorem~\ref{thm:132,213,231,312}  \\
        \hline
        $\{123,132,213,231,312\}$& $1$, for $n \ge 3$& N/A & Theorem~\ref{thm:123,132,213,231,312}\\
        \hline
    \end{tabular}
\caption{A summary of the enumeration of nonnesting permutations avoiding subsets of $\cS_3$ of size at least~2, as well as two classes of nonnesting permutations avoiding a single pattern. The formulas are valid for $n\ge1$ unless otherwise stated. OEIS refers to the Online Encyclopedia of Integer Sequences~\cite{OEIS}, OGF stands for ordinary generating function, and $\Fib_n$ denotes the $n$th Fibonacci number.}
\label{tab:length3}
\end{table}

\subsection{Avoiding one pattern}

By reversal and complementation, the enumeration of nonnesting permutations avoiding a single pattern in $\cS_3$ reduces to the enumeration of the sets $\cC_n(123)$ and $\cC_n(132)$. After we stated these as open problems in a previous version of this article posted online, a functional equation for the generating function for $\cc_n(132)$ has been found very recently by \cite{archer_laudone}. We still do not have a formula for $\cc_n(123)$. For comparison, the enumeration of quasi-Stirling permutations avoiding a simgle pattern in $\cS_3$ was also left as an open problem in~\cite{archer}, and recently solved in~\cite{archer_laudone} for the pattern~$132$.

It is easier, however, to enumerate nonnesting permutations avoiding a pattern of length 3 with repeated letters. Disregarding the trivial case of the pattern $111$, which is avoided by all nonnesting permutations, reversal and complementation reduces this problem to the enumeration of the sets $\cC_n(112)$ and $\cC_n(121)$, which are treated below. 

\begin{theorem} \label{thm:112}
For all $n\ge1$, we have $\cc_n(112)=\Cat_n$.
\end{theorem}

\begin{proof}
For any $1\le i<j\le n$, the subsequence consisting of entries $i$ and $j$ in $\pi\in\cC_n(112)$ must be either $jjii$ or $jiji$, since otherwise they would either create a nesting or an occurrence of $112$. It follows that, in the interpretation of nonnesting permutations as nonnesting matchings with labeled arcs, a permutation avoids $112$ if and only if the arcs are labeled so that their left endpoints (namely, the first occurrence of each value) appear in decreasing order. Thus, there is exactly one possible labeling of each nonnesting matching. It follows that $\cc_n(112)$ is simply the number of nonnesting matchings of $[2n]$, which is~$\Cat_n$.
\end{proof}

\begin{theorem} \label{thm:121}
    For all $n\ge1$, we have $\cc_n(121)=n!$.
\end{theorem}

\begin{proof}
For any $1\le i<j\le n$, the subsequence consisting of entries $i$ and $j$ in $\pi\in\cC_n(121)$ must be either $jjii$ or $iijj$, since otherwise they would either create a nesting or an occurrence of $121$. This forces repeated entries in $\pi$ to occur next to each other. Thus $\pi$ is obtained from a permutation in $\cS_n$ by simply duplicating each entry. This leaves $n!$ possibilities.
\end{proof}

Note that $\cC_n(112)=\cC_n(122)$, and that $\cC_n(121)=\cC_n(212)$. The latter set is the intersection of nonnesting permutations and Stirling permutations.

\subsection{Avoiding two patterns}

We start with the simple case where the two avoided patterns are monotonic.

\begin{theorem}\label{thm:increasing,decreasing}
    For all $k,\ell\ge2$ and $n \ge (k-1)(\ell-1)+1$, we have $\cc_n(12\dots k,\ell\dots21)=0$.
\end{theorem}

\begin{proof}
In any element of $\cC_n$, the subsequence obtained by deleting one copy of each entry is a permutation in $\cS_n$. By~\cite{ErdosSzekeres}, if $n\ge (k-1)(\ell-1)+1$, every permutation in $\cS_n$ must contain either an increasing subsequence of length $k$ or a decreasing subsequence of length $\ell$, that is, one of the patterns $12\dots k$ or $\ell\dots21$. It follows that $\cc_n(12\dots k,\ell\dots21)=0$. 
\end{proof}

The following result is an immediate consequence of the above theorem. It will save us some work when classifying nonnesting permutations avoiding larger sets of patterns.

\begin{corollary}\label{cor:123,321}
    For any $\Lambda\subseteq\cS_3$ such that $\{123,321\}\subseteq \Lambda$, we have $\cc_n(\Lambda)=0$ for all $n\ge5$.
\end{corollary}

\begin{theorem}\label{thm:123,231}
For all $n\ge2$, we have $$\cc_n(123,231)=\frac{n^2+5n-6}{2}.$$
\end{theorem}

\begin{proof}
Let $\Lambda = \{123,231\}$. 
We can write $\pi \in \cC_n(\Lambda)$ uniquely as $\pi= \alpha 1 \beta 1 \gamma$ for some $\alpha,\beta,\gamma$. Since $\pi$ avoids $123$ and $231$, the words $\alpha$, $\beta$ and $\gamma$ must be weakly decreasing, and we must have $\alpha \geq \beta$ and $\beta \geq \gamma$. In particular, $\card{\Set(\alpha) \cap \Set(\beta)} \leq 1$ and $\card{\Set(\beta) \cap \Set(\gamma)} \leq 1$. Note also that, since $\pi$ is nonnesting, $\beta$ cannot have repeated entries, and $\Set(\alpha) \cap \Set(\gamma)=\emptyset$. It follows that $\beta$ must have length at most $2$, leaving four cases: 
\begin{enumerate}[(1)]
\item $\pi=\alpha11\gamma$,
\item $\pi = \alpha' i1i1 \gamma$ for some $i\in\{2,3,\dots,n\}$,
\item $\pi = \alpha 1i1i \gamma'$ for some $i\in\{2,3,\dots,n\}$,
\item $\pi = \alpha' (i+1) 1 (i+1)i 1 i \gamma'$ for some $i\in\{2,3,\dots,n-1\}$.
\end{enumerate}

In case (1), both $\alpha$ and $\gamma$ must consist of decreasing sequences of double entries, and $\pi$ is uniquely determined by $\Set(\gamma)$, since $\Set(\alpha)=\{2,3,\dots,n\}\setminus\Set(\gamma)$. Additionally, in order for $\pi$ to avoid $231$, the elements of $\Set(\gamma)$ must be consecutive. Thus, either $\gamma$ is empty, or $\Set(\gamma)=\{i+1,i+2,\dots,j\}$ for some $1\le i< j\le n$. It follows that there are $1+\binom{n}{2}$ permutations in case~(1). 

In case (2), $i$ must be smaller than all the entries in $\alpha'$ and larger than all the entries in $\gamma$. Thus, $\pi$ is determined by the choice of $i\in\{2,3,\dots,n\}$, so there are $n-1$ permutations in this case. A similar argument shows that there are $n-1$ permutations in case~(3).

In case (4), $\pi$ is determined by the choice of $i\in\{2,3,\dots,n-1\}$, so there are $n-2$ permutations in this case.

Adding the number of permutations in all four cases, we obtain 
\[ \cc_n (\Lambda) = 1+\binom{n}{2} + (n-1) + (n-1) + (n-2) = \frac{n^2+5n-6}{2}.\qedhere\] 
\end{proof}

We will use $\Fib_n$ to denote the $n$th Fibonacci number, with the convention $\Fib_0=\Fib_1=1$. 

\begin{theorem}\label{thm:132,213}
For all $n \geq 0$, we have $\cc_n(132,213)=\Fib_n^2$.
\end{theorem}

\begin{proof}
Let $\Lambda = \{132,213\}$. Writing permutations  $\pi\in\cC_n(\Lambda)$ as $\pi=\alpha n\beta n\gamma$ for some words $\alpha,\beta,\gamma$, we can separate them into four cases:
\begin{enumerate}[(1)]
\item $\alpha=\beta=\emptyw$,
\item $\alpha=\emptyw$ and $\beta\neq\emptyw$,
\item $\alpha\neq\emptyw$ and $\beta=\emptyw$,
\item $\alpha\neq\emptyw$ and $\beta\neq\emptyw$.
\end{enumerate}
Denote the number of permutations in each case by $\one_n$, $\two_n$, $\three_n$ and $\four_n$, respectively, so that $\cc_n(\Lambda)=\one_n+\two_n+\three_n+\four_n$.

To obtain recurrence relations for these numbers, we consider the possible ways to generate a permutation in $\cC_{n+1}(\Lambda)$ by inserting two entries $n+1$ in a permutation in $\cC_n(\Lambda)$ from each of the above four cases.

In case~(1), there are four ways to insert two entries $n+1$ in $nn\gamma$ without creating nestings or occurrences of $213$, namely $(n+1)(n+1)nn\gamma$, $(n+1)n(n+1)n\gamma$, $nn(n+1)(n+1)\gamma$ and $n(n+1)n(n+1)\gamma$, yielding a permutation in each of cases (1), (2), (3) and (4), respectively.

In case~(2), each permutation $n\beta n\gamma$ generates two permutations $(n+1)(n+1)n\beta n\gamma$ and $(n+1)n(n+1)\beta n\gamma$, which belong to cases (1) and (2), respectively.

In case~(3), each permutation $\alpha nn\gamma$ generates two permutations $(n+1)(n+1)\alpha nn\gamma$ and $\alpha nn(n+1)(n+1)\gamma$, which belong to cases (1) and (3), respectively.

In case~(4), each permutation $\alpha n\beta n\gamma$ generates one permutation $(n+1)(n+1)\alpha n\beta n\gamma$, in case~(1). Indeed, inserting an $n+1$ anywhere after the first entry of $\alpha$ and before the second $n$ would create a $132$, whereas inserting it anywhere after the second $n$ would create a $213$.

Keeping track of how many permutations in $\cC_{n+1}(\Lambda)$ of each type are generated in each case, we conclude that 
\begin{align*}
 \one_{n+1}&=\one_n+\two_n+\three_n+\four_n=\cc_n(\Lambda),\\
 \two_{n+1}&=\one_n+\two_n,\\
 \three_{n+1}&=\one_n+\three_n,\\
 \four_{n+1}&=\one_n=\cc_{n-1}(\Lambda),
\end{align*}
from where
\begin{align*}
    \cc_{n+1}(\Lambda)&=\one_{n+1}+\two_{n+1}+\three_{n+1}+\four_{n+1}\\
    &=4\one_n+2\two_n+2\three_n+\four_n\\
    &=2(\one_n+\two_n+\three_n+\four_n)+2\one_n-\four_n\\
    &=2\cc_n(\Lambda)+2\cc_{n-1}(\Lambda)-\cc_{n-2}(\Lambda),
\end{align*}
with initial conditions $\cc_0(\Lambda)=\cc_1(\Lambda)=1$. This is the same recurrence satisfied by the squared Fibonacci numbers:
\begin{align*}
\Fib_{n+1}^2 & = (\Fib_{n} + \Fib_{n-1})^2\\
&= \Fib_{n}^2 + 2\Fib_{n}\Fib_{n-1} + \Fib_{n-1}^2\\  
&= \Fib_{n}^2 + (\Fib_{n-1}+\Fib_{n-2})\Fib_{n-1}+\Fib_n(\Fib_n-\Fib_{n-2}) + \Fib_{n-1}^2\\
&= 2\Fib_{n}^2 + 2\Fib_{n-1}^2 +\Fib_{n-2}\Fib_{n-1}-\Fib_n\Fib_{n-2}\\
&= 2\Fib_{n}^2 + 2\Fib_{n-1}^2 - \Fib_{n-2}^2.\qedhere
\end{align*}
\end{proof}

\begin{theorem}\label{thm:132,231}
For all $n \geq 2$, we have $\cc_n(132,231)=2^n$.
\end{theorem}

\begin{proof}
Let $\Lambda = \{132,231\}$.
The formula clearly holds for $n=2$, since $\cc_2(\Lambda) = 4$, so let us assume that $n\ge3$. To generate an element in $\cC_n(\Lambda)$ from $\pi\in\cC_{n-1}(\Lambda)$, the only locations to insert $n$ without creating an occurrence of $132$ or $231$ are at the very beginning or at the very end. If we also want to avoid nestings, both entries $n$ have to be inserted in the same location. This gives the recurrence $\cc_n(\Lambda) = 2\cc_{n-1}(\Lambda)$, which implies $\cc_n(\Lambda)=2^n$.
\end{proof}

\begin{theorem}\label{thm:132,312}
For all $n \geq 2$, we have $\cc_n(132,312)=4\cdot 3^{n-2}$. 
\end{theorem}

\begin{proof}
Let $\Lambda = \{132,312\}$, and assume that $n\ge3$.
Any permutation in $\cC_n(\Lambda)$ must end with $1$ or $n$, otherwise the last entry would be part of an occurrence of $132$ or $312$. Since complementation respects avoidance of $\Lambda$, it gives a bijection between permutations in $\cC_n(\Lambda)$ that end with $n$, and those that end with $1$. It follows that the number of permutations in each of the two cases is $\cc_n(\Lambda)/2$.

If $\pi \in \cC_n(\Lambda)$ ends with $n$, we can write $\pi=\alpha n\beta n$. Avoidance of $132$ forces $\alpha\ge\beta$, and the nonnesting condition prevents $\beta$ from having repeated entries. It follows that $\beta=\emptyw$ or $\beta=1$. In the first case, $\alpha$ can be any element of $\cC_{n-1}(\Lambda)$, so there are $\cc_{n-1}(\Lambda)$ such permutations. In the second case, we can write $\pi=\alpha_1 1\alpha_2 n1n$. Removing the two copies of $n$ gives a bijection between such permutations and the set of permutations in $\cC_{n-1}(\Lambda)$ that end with a $1$, so there are $\cc_{n-1}(\Lambda)/2$ such permutations.

We obtain the recurrence 
$$\cc_n(\Lambda)/2=\cc_{n-1}(\Lambda)+\cc_{n-1}(\Lambda)/2,$$
from where $\cc_n(\Lambda)=3\cc_{n-1}(\Lambda)$. Using the initial condition $\cc_{2}(\Lambda)=4$, the result follows.
\end{proof} 

\begin{theorem}\label{thm:123,213}
For all $n \geq 2$, we have $\cc_n(123,213)=4\cdot 3^{n-2}$. 
\end{theorem}

\begin{proof}
Let $\Lambda = \{123,213\}$, and let $n\ge2$. Write $\pi \in \cC_n(\Lambda)$ as $\pi =\alpha n \beta n \gamma$. To avoid both $123$ and $213$, we must have $\card{\Set(\alpha)\cup\Set(\beta)} \leq 1$. This condition, together with the fact that $\beta$ cannot have repeated letters to avoid a nesting, leaves four cases:
\begin{enumerate}[(1)]
    \item $\pi=nn\gamma$,
    \item $\pi=iinn\gamma$ for some $i\in[n-1]$,
    \item $\pi=inin\gamma$ for some $i\in[n-1]$,
    \item $\pi=nin\gamma$ for some $i\in[n-1]$.
\end{enumerate}

Denote the number of permutations in each case by $\one_n$, $\two_n$, $\three_n$ and $\four_n$, respectively, so that $\cc_n(\Lambda)=\one_n+\two_n+\three_n+\four_n$. To obtain recurrence relations, we look at the possible ways that a permutation in each of these cases could be generated by inserting two entries $n$ into a permutation in $\cC_{n-1}(\Lambda)$.

In case~(1), the word $\gamma$ is an arbitrary permutation in $\cC_{n-1}(\Lambda)$, so
\begin{equation}\label{eq:a=s}
\one_n=\cc_{n-1}(\Lambda).
\end{equation}

In case~(2), after removing the entries $n$, the permutation $ii\gamma$ is an arbitrary permutation in $\cC_{n-1}(\Lambda)$ starting with a double letter, that is, any permutation from cases~(1) and~(2). It follows that 
\begin{equation}\label{eq:b=a+b} 
\two_n=\one_{n-1}+\two_{n-1}.\end{equation}
The same is true in case~(3), so $\three_n=\two_n$.

In case~(4), the word $i \gamma$ obtained after removing the entries $n$ is an arbitrary permutation in $\cC_{n-1}(\Lambda)$, so $\four_n=\cc_{n-1}(\Lambda)$.

From equation~\eqref{eq:a=s} and the fact that $\one_n=\four_n$ and $\two_n=\three_n$, we get
\begin{equation}\label{eq:a=2a+2b} \one_n=\cc_{n-1}(\Lambda)=\one_{n-1}+\two_{n-1}+\three_{n-1}+\four_{n-1}=2\one_{n-1}+2\two_{n-1}
\end{equation}
for $n\ge3$. Combining this with equation~\eqref{eq:b=a+b}, we see that $\one_n=2\two_n$ for $n\ge3$. Using this fact in equation~\eqref{eq:a=2a+2b}, we obtain
$$\one_n=2\one_{n-1}+\one_{n-1}=3\one_{n-1}$$
for $n\ge4$, or equivalently, $$\cc_{n-1}(\Lambda)=3\cc_{n-2}(\Lambda).$$
The stated result follows from this recurrence, using the initial condition $\cc_2(\Lambda)=4$.
\end{proof} 

\begin{remark}\label{rem:nonnesting} A natural bijection between noncrossing and nonnesting permutations of $\nn$ is obtained by applying the bijection between noncrossing and nonnesting matchings described in~\cite{athanasiadis}, which preserves the left endpoints of the arcs, and labeling the arcs according to these left endpoints. It can be shown that avoidance of the pair of patterns $\{321,312\}$ is preserved by this bijection, and so Theorem~\ref{thm:123,213} above is equivalent to \cite[Thm.~4.4]{archer}.
\end{remark}

\subsection{Avoiding three patterns}

For avoidance of sets of three patterns of length~3, the case analysis is often similar to the proofs in the previous subsection.

\begin{theorem}\label{thm:123,132,213}
The ordinary generating function for nonnesting permutations that avoid $\{123, 132, 213\}$ is
\[
\sum_{n\ge0} \cc_n(123, 132, 213)\,x^n = \frac{1-x}{1-2x-2x^2+2x^3}.\]
\end{theorem}

\begin{proof}
Let $\Lambda = \{123,132,213\}$, and let $n\ge2$. Write $\pi \in \cC_n(\Lambda) $ as $ \pi =\alpha n \beta n \gamma$. As in the proof of Theorem~\ref{thm:123,213}, avoidance of $123$ and $213$ implies that $\Set(\alpha) \cup \Set(\beta)$ is either empty or consists of one element, which must be $n-1$ in order for $\pi$ to also avoid $132$. We have the same four cases as in the proof of Theorem~\ref{thm:123,213}, but now we set $i=n-1$ in all of them.

In case~(1), $\gamma$ can be any permutation in $\cC_{n-1}$, and in each of cases~(2) and~(3), $\gamma$ can be any permutation in $\cC_{n-2}(\Lambda)$. Let $A^{(4)}_n$ be the set of permutations in case~(4), and let $\four_n=|A^{(4)}_n|$. By the above decomposition,
\begin{equation}\label{eq:cc_d}
\cc_n(\Lambda)=\cc_{n-1}(\Lambda)+2\cc_{n-2}(\Lambda)+\four_n.
\end{equation}

Removing the two copies of $n$ from $\pi\in A^{(4)}_n$ produces a bijection between $A^{(4)}_n$ and the set of permutations in $\cC_{n-1}(\Lambda)$ that start with the largest entry, namely those from cases~(1) or~(4). Since the number of elements in $\cC_{n-1}(\Lambda)$ in case~(1) are counted by $\cc_{n-2}(\Lambda)$, it follows that
\begin{equation}\label{eq:d_cc} \four_n=\cc_{n-2}(\Lambda)+\four_{n-1}.
\end{equation}

Shifting the indices in equation~\eqref{eq:cc_d} down by one and solving for $\four_{n-1}$, we get
$\four_{n-1}=\cc_{n-1}(\Lambda)-\cc_{n-2}(\Lambda)-2\cc_{n-3}(\Lambda)$. Substituting in equation~\eqref{eq:d_cc}, we obtain
\[ \four_n=\cc_{n-1}(\Lambda)-2\cc_{n-3}(\Lambda).
\]
Finally, using this expression in equation~\eqref{eq:cc_d} yields the recurrence
$$\cc_n(\Lambda)=2\cc_{n-1}(\Lambda)+2\cc_{n-2}(\Lambda)-2\cc_{n-3}(\Lambda)$$
for $n\ge3$, with initial conditions $\cc_{0}(\Lambda)=\cc_{1}(\Lambda)=1$ and $\cc_{2}(\Lambda)=4$.
This recurrence is equivalent to the stated generating function.
\end{proof} 

The sequence $\cc_n(123, 132, 213)$ appears in~\cite{OEIS} as sequence A052528, although with a very different interpretation. Specifically, as shown by~\cite{hoang}, it counts vertex-transitive cover graphs of lattice quotients of essential lattice congruences of the weak order on $\cS_{n+1}$.

\begin{theorem}\label{thm:123,213,312}
For all $n\ge2$, we have $\cc_n(123, 213, 312) = n+2$.
\end{theorem}

\begin{proof}
Let $\Lambda = \{123, 213, 312\}$, and let $n\ge2$. Write $\pi \in \cC_n(\Lambda)$ as $\pi =\alpha n \beta n \gamma$. 
To avoid both $123$ and $213$, we must have $\card{\Set(\alpha) \cup \Set(\beta)}\leq 1$. 
Additionally, to avoid 312, $\gamma$ must be weakly decreasing, and $\beta \geq \gamma$. Combined with the nonnesting condition, this leaves four possibilities:
\begin{align*}
 \pi&=nn(n-1)(n-1)\dots11,\\
 \pi&=n(n-1)n(n-1)\,(n-2)(n-2)(n-3)(n-3)\dots11,\\
 \pi&=(n-1)n(n-1)n\,(n-2)(n-2)(n-3)(n-3)\dots11, \text{ or}\\
  \pi&=ii\,nn\,(n-1)(n-1)(n-2)(n-2)\dots\widehat{ii}\dots11
\end{align*}
for some $i\in[n-1]$, where we use $\widehat{ii}$ to indicate that we are skipping these entries.
We conclude that
\[\cc_n(\Lambda) = 1+1+1+(n-1)=n+2.\qedhere\]
\end{proof}

\begin{theorem}\label{thm:132,213,312}
For all $n\ge2$, we have $\cc_n(132, 213, 312) = n+2$.
\end{theorem}

\begin{proof}
Let $n\ge2$, and write $\pi \in \cC_n(132, 213, 312)$ as $\pi =\alpha n \beta n \gamma$. 
To avoid $213$, $\alpha$ and $\beta$ must be weakly increasing, and $\alpha\le\beta$. To avoid $312$, $\beta$ and $\gamma$ must be weakly decreasing, and $\beta\ge\gamma$.
Additionally, for $\pi$ to avoid $132$, we must have $\alpha\ge\beta$ and $\alpha>\gamma$, where the strictness comes from the nonnesting condition.
The requirement that $\beta$ is both weakly increasing and weakly decreasing, along with the nonnesting condition, implies that $\beta$ has length at most one.

If $\beta=\emptyw$, the above conditions on $\alpha$ and $\gamma$ imply that 
$$\pi=ii(i+1)(i+1)\dots nn\,(i-1)(i-1)(i-2)(i-2)\dots11$$
for some $i\in [n]$, giving $n$ different permutations.

If $\beta$ has length one, then the requirements $\alpha\le\beta$ and $\alpha\ge\beta$ imply that $\alpha=\emptyw$ or $\alpha=\beta$. Additionally, the condition $\beta\ge\gamma$ implies that $\beta=n-1$ in this case. Since $\gamma$ is weakly decreasing, this leaves the two possibilities
\begin{align*}
\pi&=n(n-1)n(n-1)\,(n-2)(n-2)(n-3)(n-3)\dots11,\\
\pi&=(n-1)n(n-1)n\,(n-2)(n-2)(n-3)(n-3)\dots11,
\end{align*}
for a total of $n+2$ permutations.
\end{proof} 

\begin{theorem}\label{thm:123,231,312}
For all $n\ge3$, we have $\cc_n(123,231,312) = n$.
\end{theorem}

\begin{proof}
Let $n\ge2$, and write $\pi \in \cC_n(123,231,312)$ as $\pi =\alpha n \beta n \gamma$. To avoid $123$, $\alpha$ and $\beta$ must be weakly decreasing, and $\alpha\ge\beta$. To avoid $312$, $\beta$ and $\gamma$ must be weakly decreasing, and $\beta\ge\gamma$.  To avoid $231$, we must have $\alpha\le\beta$, $\beta\le\gamma$, and $\alpha<\gamma$, using also the nonnesting condition.

If $\beta=\emptyw$, the fact that $\alpha$ and $\gamma$ are weakly decreasing, along with the inequality $\alpha<\gamma$, implies that 
$$\pi=(i-1)(i-1)(i-2)(i-2)\dots11\,nn(n-1)(n-1)\dots ii$$
for some $i\in [n]$, giving $n$ different permutations.

If $\beta\neq\emptyw$, let $i$ be an entry in $\beta$, and note that the nonnesting condition requires that the other copy of $i$ appears in $\alpha$ or $\gamma$.
This forces $\Set(\alpha)\cup\Set(\beta)\cup\Set(\gamma)=\{i\}$, because if some $j\neq i$ was in this set, then one of the conditions $\alpha\le\beta$, $\alpha\ge\beta$, $\beta\le\gamma$, or $\beta\ge\gamma$ would be violated. This can only happen if $n=2$, and $\pi$ must be one of the permutations $1212$ or $2121$ in this case.
\end{proof}

\begin{theorem}\label{thm:123,213,231}
For all $n\ge2$, we have $\cc_n(123,213,231) = 4(n-1)$.
\end{theorem}

\begin{proof}
Let $\Lambda=\{123,213,231\}$ and let $n\ge3$.
Write $\pi \in \cC_n(\Lambda) $ as $ \pi =\alpha n \beta n \gamma$. As in the proof of Theorem~\ref{thm:123,213}, avoidance of $123$ and $213$ implies that $\Set(\alpha) \cup \Set(\beta)$ is either empty or consists of one element, which must be $1$ in order for $\pi$ to also avoid $231$. This leaves the four cases from the proof of Theorem~\ref{thm:123,213}, where now we set $i=1$.

In case~(1), $\gamma$ can be any element of $\cC_{n-1}(\Lambda)$, so there are $\cc_{n-1}(\Lambda)$ permutations.

In cases~(2) and~(3), avoidance of $123$ forces $\gamma$ to be weakly decreasing, resulting in the two permutations
\begin{align*}
\pi=11nn(n-1)(n-1)\dots22,\\
\pi=1n1n(n-1)(n-1)\dots22.
\end{align*}

In case~(4), we can write $\pi=n1n\gamma_1 1\gamma_2$, where $\gamma_1\gamma_2$ is weakly decreasing, since $\pi$ avoids $123$. The nonnesting condition prevents $\gamma_1$ from having repeated letters, so $\gamma_1=\emptyw$ or $\gamma_1=n-1$ (the latter assumes that $n\ge3$), resulting in the two permutations
\begin{align*}
\pi=n1n1(n-1)(n-1)(n-2)(n-1)\dots 22,\\
\pi=n1n(n-1)1(n-1)\,(n-2)(n-2)\dots22.
\end{align*}

Combining all the cases, we obtain the recurrence
$$\cc_{n}(\Lambda)=\cc_{n-1}(\Lambda)+4$$
for $n\ge3$. Using the initial condition $\cc_{2}(\Lambda)=4$, the result follows.
\end{proof}

\subsection{Avoiding four or five patterns}

There are three cases of sets $\Lambda\subseteq\cS_3$ of size~4 and one case of size~5 that are not covered by Corollary~\ref{cor:123,321}. In all of them, the number of nonnesting permutations of $\nn$ avoiding $\Lambda$ is constant for $n\ge3$.

\begin{theorem} \label{thm:123,132,213,231}
    For all $n \geq 2$, we have $\cc_n(123,132,213,231)=4$.
\end{theorem}

\begin{proof}
Let $\Lambda = \{123,132,213,231\}$. For $n\ge3$, any $\pi \in \cC_n(\Lambda)$ must be of the form $nn\alpha$, since the avoidance condition requires that in any subsequence $\pi_i\pi_j\pi_k$ of distinct letters, $\pi_i$ must be the largest.
Therefore, $\cc_n(\Lambda)=\cc_{n-1}(\Lambda)$ for $n\ge3$. Since $\cc_2(\Lambda)=4$, the result follows. 
\end{proof}

\begin{theorem} \label{thm:123,132,231,312}
    For all $n \geq 3$, we have $\cc_n(123,132,231,312)=2$.
\end{theorem}

\begin{proof}
Let $\Lambda=\{123,132,231,312\}$ and let $n\ge3$. 
Avoidance of $132$ and $231$, together with the nonnesting condition, implies that any $\pi \in \cC_n(\Lambda)$ must be of the form $\alpha nn$ or $nn \gamma$. Additionally, avoidance of $123$ and $312$ forces $\alpha$ and $\gamma$ to be weakly decreasing. Thus, for $n\ge3$,
$$\cC_n(\Lambda) = \{(n-1)(n-1)(n-2)(n-2)\dots11nn, nn(n-1)(n-1)\dots11\}.\qedhere$$
\end{proof}

\begin{theorem} \label{thm:132,213,231,312}
    For all $n \geq 3$, we have $\cc_n(132,213,231,312)=2$.
\end{theorem}

\begin{proof}
Let $\Lambda=\{132,213,231,312\}$. Any subsequence of $\pi \in \cC_n(\Lambda)$ of length~3 with distinct entries must be increasing or decreasing. Hence, for $n\ge3$, we have
$\cC_n(\Lambda)=\{1122\dots nn,nn\dots 2211\}$.
\end{proof}

\begin{theorem} \label{thm:123,132,213,231,312}
    For all $n \ge 3$, we have $\cc_n(123,132,213,231,312)=1$.
\end{theorem}

\begin{proof}
Any subsequence of $\pi \in \cC_n(123,132,213,231,312)$ of length~3 with distinct entries must be weakly decreasing. Hence, for $n\ge3$, the only possibility is $\pi=nn\dots 2211$.
\end{proof}

\section{Some patterns of length~4}\label{sec:length4}

In this section we give a few results about nonnesting permutations avoiding sets of patterns of length~4. We do not systematically analyze all sets, but rather we introduce some tools and provide a sample of results for which the enumeration sequences are interesting. We focus on patterns where one letter is repeated, and often appearing in adjacent positions.
Tables~\ref{tab:length4-atmost2} and~\ref{tab:length4-3ormore} list sets $\Lambda$ of patterns with repeated letters (up to reversal and complementation) for which we have found a formula for $\cc_n(\Lambda)$. Patterns where the repeated letters are adjacent are colored according to the permutation in $\cS_3$, up to reverse-complement, obtained when removing one of the repeated letters: \colI{red} for $123$, \colII{orange} for $321$, \colU{blue} for $132$ and $213$, and \colD{violet} for $231$ and $312$.

\begin{table}
\centering
\begin{tabular}{|c|c|c|c|}
        \hline
        $\Lambda$ & Formula for $\cc_n(\Lambda)$ & OEIS code & Result in the paper \\
        \hline\hline
        $\{\colI{\colI{1223}}\}$ & \multirow{2}{*}{$\Cat_n^2$}  & \multirow{2}{*}{A001246} & \multirow{2}{*}{Theorem~\ref{thm:ijjk}}\\
        \cline{1-1}
        $\{\colU{1332}\}$ & &  & \\
        \hline
        \begin{tabular}{c} $\{\sigma,\tau\}$, where \\ $\sigma\in\{\colI{1123},\colI{1223},\colI{1233}\}$, \\
        $\tau\in\{\colII{3321},\colII{3221},\colII{3211}\}$ 
        \end{tabular}        
        & $0$, for $n\ge5$ & N/A & Theorem~\ref{thm:1223,3221} \\
        \hline
        $\{\colI{1223},\colU{1332}\}$ & \multirow{6}{*}{$2^{n-1} \Cat_n$} & \multirow{6}{*}{A003645} & \multirow{4}{*}{Theorem~\ref{thm:ijjk}} \\
        \cline{1-1}
        $\{\colU{1332},\colU{2113}\}$ & & & \\
        \cline{1-1}
        $\{\colU{1332},\colD{2331}\}$ & & & \\
        \cline{1-1}
        $\{\colU{1332},\colD{3112}\}$ & & & \\
        \cline{1-1} \cline{4-4}
        $\{\colI{1123},\colU{1132}\}$ & & & \multirow{2}{*}{Theorem~\ref{thm:ncat}} \\
        \cline{1-1}
        $\{\colU{1322},\colD{3122}\}$ & & & \\
        \hline
        $\{\colI{1223},\colD{2331}\}$ & $\displaystyle\left(\binom{n}{2}+1\right)\Cat_n$ & N/A & Theorem~\ref{thm:ijjk} \\
        \hline
        $\{\colU{1322},\colD{2231}\}$ & $\displaystyle\binom{2n}{n}-2^{n-1}$ & A085781 & Theorem~\ref{thm:1322,2231} \\
		\hline
        $\{1231,1321\}$    & EGF: \ $\dfrac{2}{3-e^{2x}}$  & A122704 & Theorem~\ref{thm:1231,1321} \\
        \hline
    \end{tabular}
\caption{A summary of our results enumerating nonnesting permutations avoiding some sets of one or two patterns patterns of length~4.
EGF stands for exponential generating function.}
\label{tab:length4-atmost2}
\end{table}

\begin{table}
\centering
\resizebox{\textwidth}{!}{\begin{tabular}{|c|c|c|c|}
        \hline
        $\Lambda$ & Formula for $\cc_n(\Lambda)$ & OEIS code & Result in the paper \\
        \hline\hline
        $\{\colI{1223},\colU{1332},\colU{2113}\}$ & \multirow{2}{*}{$\Fib_n\Cat_n$} & \multirow{2}{*}{A098614} & Theorem~\ref{thm:ijjk}  \\
        \cline{1-1}\cline{4-4}
        $\{\colI{1123},\colU{1132},\colU{2133}\}$ & & & Theorem~\ref{thm:ncat} \\
        \hline
        $\{\colI{1223},\colU{1332},\colD{2331}\}$ & \multirow{7}{*}{$\displaystyle\binom{2n}{n-1}$} & \multirow{7}{*}{A001791} & \multirow{4}{*}{Theorem~\ref{thm:ijjk}}  \\
        \cline{1-1}
        $\{\colU{1332},\colU{2113},\colD{2331}\}$ & & &\\
        \cline{1-1}
        $\{\colI{1223},\colU{1332},\colD{3112}\}$ & & &\\
        \cline{1-1}
        $\{\colI{1223},\colD{2331},\colD{3112}\}$ & & &\\
        \cline{1-1}\cline{4-4}
        $\{\colI{1123},\colU{1132},\colD{2331}\}$ & & &\multirow{3}{*}{Theorem~\ref{thm:ncat}} \\
        \cline{1-1}
        $\{\colI{1123},\colU{1132},\colD{3122}\}$ & & &\\
        \cline{1-1}
        $\{\colU{1332},\colU{2213},\colD{2231}\}$ & & &\\
        \hline
        $\{\colI{1123},\colU{1322},\colD{2331}\}$ & \multirow{5}{*}{$\Cat_{n+1}-1$} & \multirow{5}{*}{A001453} & Theorem~\ref{thm:1123,1322,2331} \\
        \cline{1-1}
        \cline{4-4}
        $\{\colU{1132},\colU{2213},\colD{2231}\}$ & & & Theorem~\ref{thm:1132,2213,2231} \\
        \cline{1-1}
        \cline{4-4}
        $\{\colI{1233},\colU{1322},\colD{3122}\}$ 
        & & & Theorem~\ref{thm:1233,1322,3122} \\        \cline{1-1}
        \cline{4-4}
        $\{\colU{1322}, \colU{2213}, \colD{2231}\}$ & & & Theorem~\ref{thm:1322,2213,2231} \\
        \cline{1-1}
        \cline{4-4}
        $\{\colI{1223}, \colD{2231}, \colD{3112}\}$ & & & Theorem~\ref{thm:1223,2231,3112} \\
        \hline
        $\{\colI{1123},\colU{1132},\colD{2311}\}$ & $\dfrac{n^3+9n^2-10n}{6}$, for $n \geq 2$ & A060488 & Theorem~\ref{thm:1123,1132,2311}  \\
        \hline
        $\{\colI{1123},\colU{1322},\colD{2311}\}$ & \multirow{2}{*}{$\dfrac{n^3+6n^2-7n+6}{6}$} & \multirow{2}{*}{A027378} & Theorem~\ref{thm:1123,1322,2311}  \\
        \cline{1-1}
        \cline{4-4}
        $\{\colU{1132},\colU{2213},\colD{2311}\}$ & & & Theorem~\ref{thm:1132,2213,2311} \\
        \hline
        $\{\colI{1233},\colU{1132},\colD{2311}\}$ & $n^2+n-1$, for $n\ge 3$ & A028387 & Theorem~\ref{thm:1233,1132,2311} \\
        \hline
        $\{\colI{1233},\colU{1322},\colD{2311}\}$ & \multirow{3}{*}{$n^2$} & \multirow{3}{*}{A000290}  & Theorem~\ref{thm:1233,1322,2311} \\
        \cline{1-1}
        \cline{4-4}
        $\{\colU{1322},\colU{2133},\colD{2311}\}$ & & &Theorem~\ref{thm:1322,2133,2311}\\
        \cline{1-1}
        \cline{4-4}
        $\{\colI{1123},\colD{2231},\colD{3312}\}$ & & & Theorem~\ref{thm:1123,2231,3312}\\
        \hline
        $\{\colU{1332},\colU{2133},\colD{2311}\}$ & \multirow{2}{*}{$2n^2-3n+2$} & \multirow{2}{*}{A084849} & Theorem~\ref{thm:1332,2133,2311}   \\
        \cline{1-1}
        \cline{4-4}
        $\{\colI{1123},\colD{2331},\colD{3312}\}$ & & &Theorem~\ref{thm:1123,2331,3312} \\
        \hline
        $\{\colU{1132},\colU{2133},\colD{2311}\}$ & \multirow{2}{*}{$n^2+n-2$, for $n\ge2$} & \multirow{2}{*}{A028552}  & Theorem~\ref{thm:1132,2133,2311}\\
        \cline{1-1}
        \cline{4-4}
        $\{\colI{1123},\colD{2311},\colD{3122}\}$ & & & Theorem~\ref{thm:1123,2311,3122}\\
        \hline
        $\{\colI{1123},\colU{1132},\colD{3312}\}$ & $\dfrac{7n^2 - 17n + 14}{2}$, for $n\ge2$ & \begin{tabular}{c} A140065 \\
(values differ by $1$)\end{tabular} & 
   Theorem~\ref{thm:1123,1132,3312} \\     
        \hline
        $\{\colI{1123},\colD{2311},\colD{3112}\}$ & $\dfrac{n^3+3n^2+8n-12}{6}$, for $n\ge2$ & \begin{tabular}{c} A341209 \\
(values differ by $1$)\end{tabular} & Theorem~\ref{thm:1123,2311,3112} \\
        \hline
        $\{\colI{1123},\colD{2311},\colD{3312}\}$ & $\dfrac{n^2+7n-10}{2}$, for $n\ge2$ & A183905 & Theorem~\ref{thm:1123,2311,3312} \\
        \hline
         $\{\colI{1223},\colD{2231},\colD{3312}\}$ & $\dfrac{n^3+2n}{3}$ & A006527 & Theorem~\ref{thm:1223,2231,3312} \\
        \hline
        $\{\colU{1132},\colD{3112},3121\}$ & $5\cdot3^{n-2} -1$, for $n \ge 2$ & A198643 & Theorem~\ref{thm:1132,3112,3121} \\
        \hline
        $\{1231,1321,\colU{2113}\}$ & OGF: \ $\dfrac{1+2x-\sqrt{1-8x+4x^2}}{6x}$ & A007564 & Theorem~\ref{thm:1231,1321,2113}\\
        \hline
         \begin{tabular}{c}
        $\{1231,1321,2132,$ \\ $2312,3123,3213\}$ \end{tabular}  & $n!\Fib_n$  & A005442 & Theorem~\ref{thm:1231,1321,2132,2312,3123,3213} \\
        \hline
    \end{tabular}}
\caption{A summary of our results enumerating nonnesting permutations avoiding some sets of three or more patterns patterns of length~4.}
\label{tab:length4-3ormore}
\end{table}

To prove some of these formulas, it will be convenient to view permutations $\pi\in\cC_n$ as labeled nonnesting matchings of $[2n]$, where there is an arc between $i$ and $j$ with label $\ell$ if $\pi_i=\pi_j=\ell$. The nonnesting condition guarantees that the order in which the left endpoints of the arcs appear is the same as the order in which their right endpoints appear, so there is a natural ordering of the arcs from left to right. The permutation in $\cS_n$ obtained when reading the labels of the arcs from left to right will be called the {\em underlying permutation} of $\pi$, and denoted by $\hpi$. Note that $\hpi$ is the subsequence of $\pi$ obtained by taking the left copy of each letter, or alternatively by taking the right copy of each letter. For example, if $\pi=1521352434\in\cC_5$ (whose matching appears on the left of Figure~\ref{fig:arc_diagram}), then its underlying permutation is $\hpi=15234\in\cS_5$.

\subsection{Patterns whose repeated letters are in the middle}
In some cases, it is possible to describe pattern-avoiding nonnesting permutations by imposing restrictions on the underlying permutation, whereas the (unlabeled) nonnesting matching is arbitrary. When this happens, the resulting formulas have a factor of $\Cat_n$ to account for the possible nonnesting matchings.

The next lemma will be useful when avoiding patterns of length~$4$ with a repeated letter in the middle, as it translates this restriction to an avoidance condition on the underlying permutation.

\begin{lemma} \label{lem:ijjk}
    Let $\pi\in\cC_n$, let $\hpi\in\cS_n$ be its underlying permutation, and let $ijk\in\cS_3$. Then $\pi$ avoids $ijjk$ if and only if $\hpi$  avoids $ijk$. 
\end{lemma}

\begin{proof}
    If $\pi$ contains $ijjk$, then the subsequence of $\pi$ consisting of the left copy of each letter must contain $ijk$. Conversely, if $\hpi$ contains $ijk$, the nonnesting condition forces the right copy of (the letter playing the role of\footnote{When $j$ is an entry in a pattern, we will often refer to ``copies of $j$'' in a permutation to mean copies of the letter playing the role of $j$ in an occurrence of the pattern.}) $j$ to appear before the right copy of $k$. It follows that $\pi$ contains $ijjk$.
\end{proof}

For a set $\Sigma\subseteq\cS_3$, we denote by $\cS_n(\Sigma)$ the set of permutations in $\cS_n$ that avoid all the patterns in $\Sigma$, and let $\cs_n(\Sigma)=|\cS_n(\Sigma)|$. The next lemma reduces the enumeration of nonnesting permutations avoiding patterns of length~$4$ with a repeated letter in the middle to the enumeration of permutations in $\cS_n$ avoiding patterns of length~$3$, which was done by 
\cite{simion_schmidt}; see also~\cite{knuth}.

\begin{lemma}\label{lem:ijjk_enumeration}
Let $\Sigma\subseteq\cS_3$, and let $\Lambda=\{\sigma_1\sigma_2\sigma_2\sigma_3:\sigma\in\Sigma\}$.
Then, for any $n\ge1$, $\cc_n(\Lambda)=\cs_n(\Sigma)\,\Cat_n$.
\end{lemma}

\begin{proof}
Let $\sigma \in \cS_3$. By Lemma~\ref{lem:ijjk}, $\pi\in\cC_n$ avoids $\sigma_1\sigma_2\sigma_2\sigma_3$ if and only if the underlying permutation $\hpi\in\cS_n$ avoids $\sigma$. 
Thus, $\pi\in\cC_n(\Lambda)$ is determined by first choosing a nonnesting matching on $[2n]$, of which there are $\Cat_n$, and then an underlying permutation $\hpi\in\cS_n(\Sigma)$, of which there are $\cs_n(\Sigma)$.
\end{proof}

\begin{theorem}\label{thm:ijjk}
For all $n\ge1$, we have 
\begin{enumerate}[(a)]
    \item $\cc_n(ijjk)=\Cat_n^2$ for every $ijk\in\cS_3$,
    \item $\cc_n(1223,1332)=\cc_n(1332,2113)=\cc_n(1332,2331)=\cc_n(1332,3112)=2^{n-1} \Cat_n$,
    \item $\cc_n(1223,2331) = \left(\binom{n}{2}+1\right)\Cat_n$,
    \item $\cc_n(1223,1332,2113)= \Fib_n\Cat_n$,
    \item $\cc_n(1223,1332,2331)=\cc_n(1332,2113,2331)=\cc_n(1223,1332,3112)=\cc_n(1223,2331,3112)= n\Cat_n=\binom{2n}{n-1}$.
\end{enumerate}
\end{theorem}

\begin{proof}
These results follow from Lemma~\ref{lem:ijjk_enumeration}, along with the following formulas: 
\begin{enumerate}[(a)]
\item $\cs_n(ijk)=\Cat_n$, as shown in~\cite{knuth};
\item  $\cs_n(123,132)=\cs_n(132,213)=\cs_n(132,231)=\cs_n(132,312)=2^{n-1}$, as shown in~\cite[Prop.~7--10]{simion_schmidt};
\item $\cs_n(123,231)= \binom{n}{2}+1$, as shown in~\cite[Prop.~11]{simion_schmidt};
\item $\cs_n(123,132,213)= \Fib_n$, by~\cite[Prop.~15]{simion_schmidt};
\item $\cs_n(123,132,231)= 
\cs_n(132,213,231)=\cs_n(123,132,312)=\cs_n(123,231,312)=n$, by~\cite[Prop.~16 and 16$^*$]{simion_schmidt}.
\qedhere
\end{enumerate}
\end{proof}

Lemma~\ref{lem:ijjk} would not hold if we replaced $ijjk$ with $iijk$. For example, $\pi=113232$ contains $1123$ but its underlying permutation $\hpi=132$ avoids $123$. To enumerate permutations avoiding patterns where the repeated letter is not in the middle, the next lemma will be useful.

\begin{lemma} \label{lem:forward}
    Let $\pi\in\cC_n$, and let $ijk\in\cS_3$. If $\pi$ avoids either $iijk$ or $ijkk$, then $\pi$ avoids $ijjk$. 
\end{lemma}

\begin{proof}
    We prove the contrapositive statement. Suppose that $\pi$ contains $ijjk$. Since $\pi$ is nonnesting, the other copy of $i$ must occur before the right copy of $j$. This creates an occurrence of $iijk$. A symmetric argument shows that $\pi$ also contains $ijkk$.
\end{proof}

\begin{theorem}\label{thm:1223,3221}
Let $\sigma\in\{1123,1223,1233\}$ and $\tau\in\{3321,3221,3211\}$. Then, for all $n\ge5$, we have $\cc_n(\sigma,\tau)=0$.    
\end{theorem}

\begin{proof}
By Lemma~\ref{lem:forward}, any $\pi \in \cC_n(\sigma,\tau)$ must avoid $1223$ and $3221$.
But then, by Lemma~\ref{lem:ijjk}, its underlying permutation $\hpi \in \cS_n$ must avoid $123$ and $321$. We know by~\cite{ErdosSzekeres} that $\cs_n(123,321)=0$ for all $n\ge 5$. Therefore, $\cc_n(\sigma,\tau)=0$ for all $n\ge 5$.
\end{proof}

The following lemma is a partial converse of Lemma~\ref{lem:forward}. 

\begin{lemma} \label{lem:reverse}
    Let $ijk\in\cS_3$, let $\Lambda = \{iikj, ikkj, ikjj, ijkj, ikjk\}$, and let $\pi \in \cC_n(ijjk)$. If $\pi$ avoids some $\sigma \in \Lambda$, then $\pi$ avoids $iijk$. 
\end{lemma}

\begin{proof}
    Again, we prove the contrapositive statement. Let $\pi \in \cC_n(ijjk)$, and suppose that $\pi$ contains $iijk$. The other copy of $j$ must be to the right of this $k$, in order to avoid $ijjk$, so $\pi$ contains $iijkj$. Now, the other copy of $k$ must be to the left of the first copy of $j$. Therefore, $\pi$ must contain either $ikijkj$ or $iikjkj$. In both cases, $\pi$ contains all the patterns in $\Lambda$. 
\end{proof}

Lemmas~\ref{lem:forward} and~\ref{lem:reverse} provide bijections between many sets of pattern-avoiding nonnesting permutations, allowing us to derive from Theorem~\ref{thm:ijjk} some formulas for patterns where the repeated letter is not in the middle.
The next theorem gives a sample of some such results, which is by no means exhaustive.

\begin{theorem} \label{thm:1123,1132} \label{thm:ncat}
For all $n\ge1$, we have \begin{enumerate}[(a)]
    \item $\cc_n(1123,1132)=\cc_n(1322,3122)=2^{n-1}\Cat_n$,
    \item $\cc_n(1123,1132,2133)= \Fib_n\Cat_n$,
    \item $\cc_n(1123,1132,2331)=\cc_n(1123,1132,3122)=\cc_n(1332,2213,2231)=\binom{2n}{n-1}.$
\end{enumerate}
\end{theorem}
\begin{proof}
We claim that $\cC_n(1123,1132)=\cC_n(1223,1332)$.
The inclusion to the right follows from Lemma~\ref{lem:forward}. For the reverse inclusion, suppose that $\pi\in\cC_n(1223,1332)$. Lemma~\ref{lem:reverse} with $ijk=123$ implies that $\pi$ avoids $1123$, and the same lemma with $ijk=132$ implies that $\pi$ avoids $1132$.
A similar argument shows that $\cC_n(1322,3122)= \cC_n(1332,3112)$.
Part~(a) now follows from Theorem~\ref{thm:ijjk}(b). 

For part~(b), one can similarly show that $\cC_n(1123,1132,2133) = \cC_n(1223,1332,2113)$ using Lemmas~\ref{lem:forward} and~\ref{lem:reverse}, and then apply Theorem~\ref{thm:ijjk}(d).

For part~(c), Lemmas~\ref{lem:forward} and \ref{lem:reverse} imply the equalities $\cC_n(1123,1132,2331) = \cC_n(1223,1332,2331)$, $\cC_n(1123,1132,3122) = \cC_n(1223,1332,3112)$, and $\cC_n(1332,2213,2231) = \cC_n(1332,2113,2331)$. The enumeration of these sets is given in Theorem~\ref{thm:ijjk}(e).
\end{proof}

\subsection{Other patterns whose repeated letters are adjacent}
The restrictions that we consider in this subsection no longer translate into restrictions for the underlying permutations.
These enumerative results often have more complicated proofs that require separating the permutations into different cases. We will often decompose permutations as follows.

\begin{lemma}\label{lem:decomposition1}
    Any $\pi \in \cC_n$ can be written as $\pi = \alpha 1 \beta 1 \gamma$, where $\beta$ has no repeated entries, and $\Set(\alpha)\cap\Set(\gamma)=\emptyset$. Thus, we have a disjoint union
$\{2,3,\dots,n\}=A\sqcup B_1\sqcup B_2\sqcup C$, where 
\begin{equation}\label{eq:ABC} 
A=\Set(\alpha)\setminus\Set(\beta), \ 
B_1=\Set(\alpha)\cap\Set(\beta), \
B_2=\Set(\gamma)\cap\Set(\beta), \ C=\Set(\gamma)\setminus\Set(\beta).
\end{equation}
Additionally, elements of $B_1$ (resp.\ $B_2$) must appear in the same order in $\beta$ as in $\alpha$ (resp.\ $\gamma$). If $\alpha$ is weakly monotone, then it consists of the elements of $A$ (each of which is duplicated) followed by the elements of $B_1$. Similarly, if $\gamma$ is weakly monotone, it consists of the elements of $B_2$ followed by the elements of $C$ (each of which is duplicated). 
\end{lemma}

\begin{proof}
The positions of the $1$s and the nonnesting condition guarantee that $\beta$ has no repeated entries, and that no entry appears in both $\alpha$ and $\gamma$.
Entries in $\beta$ that have their other copy in $\alpha$ (resp.\ $\gamma$) must appear in the same order in both subwords because of the nonnesting condition. In the special case that $\alpha$ is weakly monotone, the nonnesting condition prevents duplicated entries (those in $A$) to appear after entries in $B_1$, and similarly when $\gamma$ is weakly monotone.
\end{proof}

We will use the notation from Lemma~\ref{lem:decomposition1} throughout this section. Additionally, we let $\beta_1$ and $\beta_2$ be the subsequences of $\beta$ consisting of the elements of $B_1$ and $B_2$, respectively.

The next five theorems deal with subsets of $\cC_n(1223,1332,2331)$, the first set in Theorem~\ref{thm:ijjk}(e). Figure~\ref{fig:1223,1332,2331} shows the containment relationships between these sets as a Hasse diagram.

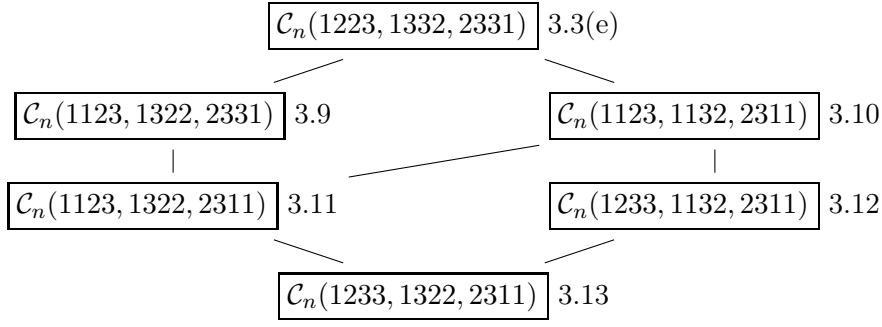
\begin{figure}[ht]
    \centering
    \begin{tikzpicture}[scale=1.2]
        \node (A) at (0,3) {\fbox{$\cC_n(1223,1332,2331)$} \ref{thm:ijjk}(e)};
        \node (B) at (-3,2) {\fbox{$\cC_n(1123,1322,2331)$}  \ref{thm:1123,1322,2331}};
        \node (C) at (3,2) {\fbox{$\cC_n(1123,1132,2311)$}  \ref{thm:1123,1132,2311}};
        \node (D) at (-3,1) {\fbox{$\cC_n(1123,1322,2311)$}  \ref{thm:1123,1322,2311}};
        \node (E) at (3,1) {\fbox{$\cC_n(1233,1132,2311)$}  \ref{thm:1233,1132,2311}};
        \node (F) at (0,0) {\fbox{$\cC_n(1233,1322,2311)$}  \ref{thm:1233,1322,2311}};
        \draw (A)--(B)--(D)--(F)--(E)--(C)--(A); \draw (C)--(D);
    \end{tikzpicture}
    \caption{The subsets of $\cC_n(1223,1332,2331)$ enumerated in this section, along with the theorem number.}
    \label{fig:1223,1332,2331}
\end{figure}

\begin{theorem}\label{thm:1123,1322,2331}
    For all $n\ge 1$, we have $\cc_n(1123,1322,2331) = \Cat_{n+1}-1$.
\end{theorem}

\begin{proof}
    We decompose $\pi \in \cC_n(1123,1322,2331)$ as in Lemma~\ref{lem:decomposition1}. 
    Since $\pi$ avoids $1322$, it must also avoid $1332$ by Lemma~\ref{lem:forward}. Now, Lemma~\ref{lem:reverse}, together with avoidance of $1123$, implies that $\pi$ avoids $1132$ as well. Avoidance of both $1123$ and $1132$ implies that $\card{\Set(\gamma)}\le1$. And since $\pi$ avoids $2331$, $\alpha\beta_1$ must avoid $122$ which, as in the proof of Theorem~\ref{thm:112}, is equivalent to its underlying permutation being decreasing. Since $\beta_1$ has no repeated letters, $\alpha1\beta_11$ avoids $122$ as well.
        
    If $\gamma = \emptyw$, then $\pi=\alpha1\beta_11$ is an arbitrary permutation in $\cC_n(122)$. Indeed, avoidance of $122$ implies avoidance of $1322$ and $2331$, and it is equivalent to avoidance of $112$, which implies avoidance of $1123$. By Theorem~\ref{thm:112}, there are $\Cat_n$ permutations in this case. 

    Now suppose that $\Set(\gamma)=\{k\}$ for some $2\le k\le n$.
    Avoidance of $1322$ requires that, in $\alpha$, any entries larger than $k$ must be to the left of any entries smaller than $k$. Thus, we can write $\alpha=\alpha_1\alpha_2$, where $\alpha_1$ and $\alpha_2$ consist of entries larger and smaller than $k$, respectively.
    
    Consider first the case when all the elements in $B_1$ are smaller than $k$. Then all the entries greater than $k$ are in $\alpha_1$, and $\st(\alpha_1)$ is an arbitrary permutation in $\cC_{n-k}(122)$. Similarly, $\st(\alpha_21\beta_11)$ is an arbitrary permutation in $\cC_{k-1}(122)$. It follows that $\alpha_21\beta1\gamma$ is an arbitrary permutation in $\cC_k$ whose underlying permutation is $(k-1)(k-2)\dots1k$. 
    Indeed, this condition on $\alpha_21\beta1\gamma$ and the fact that $\st(\alpha_1)\in\cC_{n-k}(122)$ guarantee that $\pi=\alpha_1\alpha_21\beta1\gamma$ does not contain any of the patterns $1123,1322,2331$. Since there are $\Cat_k$ permutations in $\cC_k$ with a fixed underlying permutation (one for each nonnesting matching), the number of permutations $\pi$ in this case is
    \[\sum_{k=2}^n \Cat_{n-k}\Cat_k = \Cat_{n+1}-\Cat_n-\Cat_{n-1}.\]

    Finally, consider the case when some element $b\in B_1$ is greater than $k$. Since the underlying permutation of $\alpha\beta_1$ is decreasing and $\beta_1$ has no repeated letters, $\beta_1$ is decreasing, so we can assume that $b$ is the first entry in $\beta_1$.
    Since $\pi$ avoids $1322$, the first copy of $k$ must appear before $b$, so $\beta=k\beta_1$.
   
    Let us show that $\alpha_2=(k-1)(k-2)\dots 2$. First, $\alpha_2$ cannot have repeated letters; otherwise, together with $k$ and $b$, they would form an occurrence of $1123$. Second, $\alpha_2$ contains the first occurrence of each letter in $\{2,\dots k-1\}$, and they must appear in decreasing order because otherwise $\pi$ would contain $2331$. 
    
    We claim that, in fact, 
    $$\pi=\alpha_1\,(k-1)(k-2)\dots 21\,k\,b(b-1)\dots (k+1)\,(k-1)(k-2)\dots 21\,k,$$
    where $\st(\alpha_1 1\, b(b-1)\dots(k+1)\, 1)$ is an arbitrary permutation in $\cC_{n-k+1}(122)$ not ending with $11$. Clearly, this permutation avoids $122$ (because $\pi$ avoids $2331$) and does not end with $11$ (because $b>k$). To see that it is arbitrary, note that avoidance of $122$ in this permutation guarantees that $\pi$ avoids the three patterns $1123,1322,2331$. Permutations in $\cC_{n-k+1}(122)$ that do end with $11$, by removing $11$ and standardizing, are in bijection with permutations in $\cC_{n-k}(122)$. We deduce that the number of permutations $\pi$ in this case is 
    \[\sum_{k=2}^n (\Cat_{n-k+1}-\Cat_{n-k}) = \Cat_{n-1}-1.\]

    Summing up all the cases, we have
    \[\cc_n(1123,1322,2331)=\Cat_n + (\Cat_{n+1}-\Cat_n-\Cat_{n-1})+(\Cat_{n-1}-1) = \Cat_{n+1}-1.\qedhere\]
\end{proof}

\begin{theorem} \label{thm:1123,1132,2311}
For all $n\ge2$, we have $$\cc_n(1123,1132,2311)=\frac{n^3+9n^2-10n}{6}.$$
\end{theorem}

\begin{proof}
We decompose $\pi \in \cC_n(1123,1132,2311)$, for $n\ge2$, as in Lemma~\ref{lem:decomposition1}. 
In order for $\pi$ to avoid $2311$, $\alpha$ must be weakly decreasing, and so the elements of $B_1$ must be decreasing in $\beta$.
To avoid both $1123$ and $1132$, we must have $\card{\Set(\gamma)}\leq1$, leaving the following three possibilities for $\gamma$.

If $\gamma=\emptyw$, it follows that
$$\pi=nn(n-1)(n-1)\dots (i+1) (i+1)\, i (i-1)\dots 1\,i (i-1)\dots 1$$
for some $i\in[n]$, giving $n$ permutations.

Suppose now that $\gamma=jj$ for some $j\in\{2,\dots,n\}$. If all the elements of $B_1$ are smaller than $j$, we have
\begin{equation}\label{eq:pi_skip_jj}\pi=nn(n-1)(n-1)\dots\widehat{jj}\dots(i+1)(i+1)\,i(i-1)\dots 1\,i(i-1)\dots 1\,jj\end{equation} 
for some $1\le i<j\le n$, giving $\binom{n}{2}$ permutations.
Otherwise, in order to avoid $2311$, only one element of $B_1$ can be bigger than $j$, so
\begin{equation}\label{eq:disappear_when_avoiding_1322}\pi=nn(n-1)(n-1)\dots(j+2)(j+2)\,(j+1)(j-1)\dots 1\,(j+1)(j-1)\dots 1\,jj\end{equation}
with $2\le j\le n-1$, giving $n-2$ permutations in this case. Adding these cases, the number of permutations where $\gamma=jj$ for some $j$ equals 
$$\binom{n}{2}+n-2.$$

Finally, suppose that $\gamma=j$ for some $j\in\{2,\dots,n\}$, which forces the other copy of $j$ to appear in $\beta$. 
Recall that the other entries in $\beta$ (that is, the elements of $B_1$) are decreasing.
If all these elements are smaller than $j$, then the are no restrictions on the position of $j$ inside $\beta$, and $\pi$ is obtained from equation~\eqref{eq:pi_skip_jj} by moving the first copy of $j$ and inserting it in $\beta$, in one of the $i$ available positions. Thus, the number of permutations in this case is 
$$\sum_{1\le i<j\le n}i=\binom{n+1}{3}.$$

If some elements of $B_1$ are larger than $j$, consider two subcases. If $j$ is the first entry in $\beta$, then 
\begin{equation}\label{eq:B1larger,jfirst}
    \pi=nn(n-1)(n-1)\dots(i+1)(i+1)\,i(i-1)\dots\widehat{j}\dots 1\,j\,i(i-1)\dots\widehat{j}\dots 1\,j,
\end{equation}
for some $2\le j<i\le n$, giving $\binom{n-1}{2}$ permutations. 
Otherwise, in order to avoid $2311$, there can be only one element of $B_1$ that is larger than $j$. In this case, $\pi$ is obtained from equation~\eqref{eq:disappear_when_avoiding_1322} by moving the first copy of $j$ and inserting it in $\beta$, in any of the $j-1$ positions other than the first one, giving 
\begin{equation}\label{eq:disappear_when_avoiding_1322-2}\sum_{j=2}^{n-1} (j-1)=\binom{n-1}{2}\end{equation} 
permutations. 

By adding all the cases, we have 
\[\cc_n(1123,1132,2311)=n+\binom{n}{2}+n-2+\binom{n+1}{3}+\binom{n-1}{2}+\binom{n-1}{2}=\frac{n^3+9n^2-10n}{6}.\qedhere\]
\end{proof}

\begin{theorem} \label{thm:1123,1322,2311}
    For all $n\ge1$, we have $$\cc_n(1123,1322,2311) = \dfrac{n^3+6n^2-7n+6}{6}.$$
\end{theorem}

\begin{proof}
Let us first show that \begin{equation}\label{eq:intersection0}
\cC_n(1123,1322,2311)=\cC_n(1123,1322,2331)\cap \cC_n(1123,1132,2311),
\end{equation}
that is, the intersection of the sets from Theorems~\ref{thm:1123,1322,2331} and~\ref{thm:1123,1132,2311}.

Let $\pi\in\cC_n(1123,1322,2311)$. Since $\pi$ avoids $2311$, it also avoids $2331$ by Lemma~\ref{lem:forward}. 
On the other hand, as in the proof of Theorem~\ref{thm:1123,1322,2331}, avoidance of $1123$ and $1322$ implies avoidance of $1132$. This proves the inclusion to the right in equation~\eqref{eq:intersection0}.
Conversely, a permutation in the intersection of the two sets on the right-hand side must avoid the three patterns $1123,1322,2311$.

We will adapt the proof of Theorem~\ref{thm:1123,1132,2311} by removing the two cases  where the permutation $\pi$ contains $1322$. One is when $\pi$ is given by equation~\eqref{eq:disappear_when_avoiding_1322}, accounting for $n-2$ permutations. The other the case counted in 
equation~\eqref{eq:disappear_when_avoiding_1322-2}, namely, when $\gamma=j$, there is an element in $B_1$ larger than $j$, and $j$ is not the first entry in $\beta$.

Adding the remaining cases, we get
\[\cc_n(1123,1132,2311)=n+\binom{n}{2}+\binom{n+1}{3}+\binom{n-1}{2}= \dfrac{n^3+6n^2-7n+6}{6}.\qedhere\]
\end{proof}

\begin{theorem}\label{thm:1233,1132,2311}
    For all $n\ge3$, we have $\cc_n(1233,1132,2311) = n^2+n-1$.
\end{theorem}

\begin{proof}
    Let $\pi\in\cC_n(1233,1132,2311)$. Since $\pi$ avoids $1233$, it must also avoid $1223$ by Lemma~\ref{lem:forward}. Now, Lemma~\ref{lem:reverse}, together with avoidance of $1132$, implies that $\pi$ also avoids $1123$. It follows that 
    $\cC_n(1233,1132,2311)\subseteq\cC_n(1123,1132,2311)$, the set considered in Theorem~\ref{thm:1123,1132,2311}.

    Let us show how to modify the proof of this theorem to eliminate the cases where $\pi$ contains $1233$. The case $\gamma=\emptyw$ does not change and contributes $n$ permutations. In the case $\gamma=jj$, the permutation in equation~\eqref{eq:pi_skip_jj} avoids $1233$ only if $i=1$, giving $n-1$ permutations. The permutation in equation~\eqref{eq:disappear_when_avoiding_1322} avoids $1233$ only if $j=2$, giving $1$ permutation, if we use the assumption $n\ge3$.
    
    In the case $\gamma=j$, if the elements of $B_1$ are smaller than $j$, then the other copy of $j$ has to be the first entry in $\beta$ in order to avoid $1233$, giving $\binom{n}{2}$ permutations of the form
    \begin{equation}\label{eq:pi_skip_jj_jmoved}\pi=nn(n-1)(n-1)\dots\widehat{jj}\dots (i+1)(i+1)\,i(i-1)\dots 1\,j\,i(i-1)\dots 1\,j\end{equation}
    for $1\le i<j\le n$. If some element of $B_1$ is larger than $j$, we get the $\binom{n-1}{2}$ permutations from equation~\eqref{eq:B1larger,jfirst} where $j$ is the first entry in $\beta$. If $j$ is not the first entry, then
    $$\pi=nn(n-1)(n-1)\dots(j+2)(j+2)\,(j+1)(j-1)\dots 1\,(j+1)j(j-1)\dots 1\,j$$
    for some $2\le j\le n-1$, giving $n-2$ permutations.
    
    Adding up all the cases, we get
    \[
    \cc_n(1233,1132,2311) = n + (n-1) + 1 + \binom{n}{2}+\binom{n-1}{2} + (n-2) = n^2+n-1.\qedhere
    \]
\end{proof}

\begin{theorem} \label{thm:1233,1322,2311}
For all $n\ge 1$, we have $\cc_n(1233,1322,2311) = n^2$.
\end{theorem}

\begin{proof}
Let us first show that \begin{equation}\label{eq:intersection}\cC_n(1233,1322,2311)=\cC_n(1123,1322,2311)\cap\cC_n(1233,1132,2311),
\end{equation}
that is, the intersection of the sets from Theorems~\ref{thm:1123,1322,2311} and~\ref{thm:1233,1132,2311}.

 Let $\pi\in\cC_n(1233,1322,2311)$. Since $\pi$ avoids $1233$, it must also avoid $1223$ by Lemma~\ref{lem:forward}. But avoidance of $1223$ and $1322$ implies avoidance of $1123$ by Lemma~\ref{lem:reverse}. This shows that 
 $\cC_n(1233,1322,2311)\subseteq\cC_n(1123,1322,2311)$.
Similarly, since $\pi$ avoids $1322$, it must also avoid $1332$ by Lemma~\ref{lem:forward}. But avoidance of $1332$ and $1233$ implies avoidance of $1132$ by Lemma~\ref{lem:reverse}. This shows that $\cC_n(1233,1322,2311)\subseteq\cC_n(1233,1132,2311)$.
Conversely, if a permutation is in the intersection on the right-hand side of equation~\eqref{eq:intersection}, then it clearly avoids the patterns $1233,1322,2311$.

To find $\cc_n(1233,1322,2311)$, we follow the proofs of Theorems~\ref{thm:1123,1322,2311} and~\ref{thm:1233,1132,2311}, and take the permutations that appear in both.
 When $\gamma=\emptyw$, we get the same $n$ permutations. When $\gamma=jj$, we get the $n-1$ permutation from equation~\eqref{eq:pi_skip_jj} with $i=1$. When $\gamma=j$, 
we get the $\binom{n}{2}$ permutations from equation~\eqref{eq:pi_skip_jj_jmoved} and the $\binom{n-1}{2}$ permutations from equation~\eqref{eq:B1larger,jfirst}.

    Adding up all the cases, we get
    \[
    \cc_n(1233,1322,2311) = n + (n-1) +\binom{n}{2}+\binom{n-1}{2} = n^2.\qedhere
    \]
\end{proof}

In the next six theorems, we consider subsets of $\cC_n(1332,2113,2331)$, which is the second set in Theorem~\ref{thm:ijjk}(e).
Figure~\ref{fig:1332,2113,2331} shows the containment relationships between these sets.

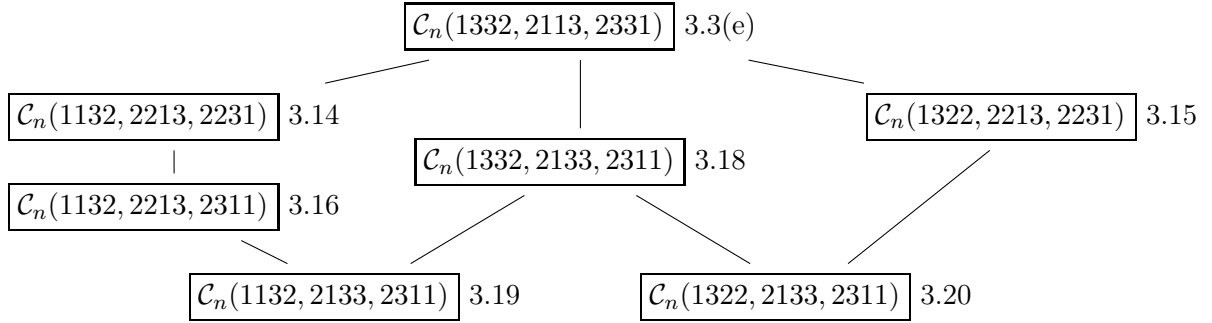
\begin{figure}[ht]
    \centering
    \begin{tikzpicture}[scale=1.1]
        \node (A) at (0,3) {\fbox{$\cC_n(1332,2113,2331)$} \ref{thm:ijjk}(e)};
        \node (B) at (-4.5,2) {\fbox{$\cC_n(1132,2213,2231)$} \ref{thm:1132,2213,2231}};
        \node (C) at (-4.5,1) {\fbox{$\cC_n(1132,2213,2311)$} \ref{thm:1132,2213,2311}};
        \node (D) at (-2.5,0) {\fbox{$\cC_n(1132,2133,2311)$} \ref{thm:1132,2133,2311}};
        \node (E) at (0,1.5) {\fbox{$\cC_n(1332,2133,2311)$} \ref{thm:1332,2133,2311}};
        \node (F) at (5,2) {\fbox{$\cC_n(1322,2213,2231)$} \ref{thm:1322,2213,2231}};
        \node (G) at (2.5,0) {\fbox{$\cC_n(1322,2133,2311)$} \ref{thm:1322,2133,2311}};
        \draw (A)--(B)--(C)--(D)--(E)--(A)--(F)--(G)--(E);
    \end{tikzpicture}
    \caption{The subsets of $\cC_n(1332,2113,2331)$ enumerated in this section.}
    \label{fig:1332,2113,2331}
\end{figure}
In the next proof, we let $\cD_n$ be the set of Dyck words of length $2n$, that is, words consisting of $n$ $\uu$s and $n$ $\dd$s with the property that no prefix contains more $\dd$s than $\uu$s. It is well known (see~\cite{Stanley_Catalan}) that $\card{\cD_n}=\Cat_n$.

\begin{theorem}\label{thm:1132,2213,2231}
    For all $n\ge 1$, we have $\cc_n(1132,2213,2231) = \Cat_{n+1}-1.$
\end{theorem}

\begin{proof}
    We decompose $\pi \in \cC_n(1132,2213,2231)$ as in Lemma~\ref{lem:decomposition1}. Avoidance of $2213$ implies avoidance of $2113$ by Lemma~\ref{lem:forward}, which requires $\alpha > \gamma$.
    Avoidance of $1132$ forces $\gamma$ to be weakly increasing, so we can write $\beta_2=23\dots i$ and $\gamma=23\dots i\,(i+1)(i+1)(i+2)(i+2)\dots jj$ for some $1\le i\le j\le n$.  
    Avoidance of $2231$ forces $\alpha\beta_1$ to avoid $112$, which implies that its underlying permutation is decreasing, so in particular $\beta_1$ is decreasing, since it consists of second copies of entries. It follows that $\st(\alpha1\beta_11)\in \cC_{n-j+1}(112)$.

    In fact, the above are the only restrictions on $\alpha$, $\beta_1$, $\beta_2$, and $\gamma$, in the sense that any choice of $1\le i\le j\le n$, any choice of $\st(\alpha1\beta_11)\in\cC_{n-j+1}(112)$, and any way to interleave the entries of $\beta_2=23\dots i$ with the entries of $\beta_1$ determines a (unique) permutation $\pi \in \cC_n(1132,2213,2231)$. To count these choices, we will describe a bijection between such permutations and certain Dyck words. Given $\pi \in \cC_n(1132,2213,2231)$ decomposed as above, construct a Dyck word as follows.
\begin{enumerate}
    \item Start with the Dyck word $w_0\in\cD_{n-j+1}$ obtained from the permutation $\st(\alpha1\beta_11)$ by simply replacing the first copy of each entry with a $\uu$ and the second copy with a $\dd$. 
    Viewing $\st(\alpha1\beta_11)$ as a nonnesting matching, this is the standard bijection between nonnesting matchings and Dyck words.
    \item Insert a $\uu\dd$ right after the last $\uu$ of $w_0$. This is the $\uu$ corresponding to the first copy of $1$, since $\beta_1$ consists only of second copies of entries. Let $w_0'$ be the resulting word in $\cD_{n-j+2}$. Note that each of the $\dd$s after this inserted $\uu\dd$ corresponds to an element of $\beta_11$. 
    \item For each entry of $\beta_2$ that is interleaved with $1\beta_11$ in $\pi$, insert a $\uu\dd$ in the corresponding location within the last run of $\dd$s in $w_0'$. Specifically, elements of $\beta_2$ that lie between the first $1$ and $\beta_1$ become $\uu\dd$s inserted right after the $\uu\dd$ from step 2, and elements of $\beta_2$ that lie between $\beta_1$ and the second $1$ become $\uu\dd$s inserted right before the last $\dd$ of $w_0'$. This step inserts a total of $i-1$ $\uu\dd$s, producing a word in $\cD_{n-j+i+1}$.
    \item Finally, append $\card{\Set(\gamma)}=j-i$ $\uu\dd$s to the end of the word, to obtain a word $w\in\cD_{n+1}$.
\end{enumerate}

We claim that the map $\pi\mapsto w$ is a bijection between $\cC_n(1132,2213,2231)$ and $\cD_{n+1}\setminus\{(\uu\dd)^{n+1}\}$, from which it will follow that $\cc_n(1132,2213,2231)=\Cat_{n+1}-1$.

First, it is clear by construction that $w\in\cD_{n+1}$, and that $w\neq(\uu\dd)^{n+1}$, since $w$ has the two consecutive $\uu$s that were created in step 2.
To see that it is a bijection, let us show that, given an arbitrary $w\in\cD_{n+1}\setminus\{(\uu\dd)^{n+1}\}$, we can uniquely recover the permutation $\pi$ that it came from. We start by finding the last two consecutive $\uu$s in $w$, which must exist because $w\neq(\uu\dd)^{n+1}$. 
Then, the word $w_0$ obtained from $w$ by removing all the pairs $\uu\dd$ to the right of the first of these two $\uu$s, determines the permutation $\st(\alpha1\beta_11)$ by simply reversing step 1.
The location of the removed pairs $\uu\dd$ determine the positions of the entries in $\beta_2$ relative to those of $\beta_1$, and the number of removed pairs $\uu\dd$ at the end of $w$ determine $\card{\Set(\gamma)}$. This information uniquely determines the permutation $\pi\in\cC_n(1132,2213,2231)$.
\end{proof}

As an example of the above bijection, let
$$\pi=12\,11\,12\,10\,9\,8\,7\,11\,\colD{1}\,\colI{2}\,10\,\colI{3\,4}\,9\,8\,\colI{5}\,7\,\colD{1}\,\colI{2\,3\,4\,5}\,\colU{6\,6}\in\cC_{12}(1132,2213,2231),$$
which has $i=5$ and $j=6$. In step~1, we have 
$$\st(\alpha \colD{1}\beta_1 \colD{1})=\st(12\,11\,12\,10\,9\,8\,7\,11\,\colD{1}\,10\,9\,8\,7\,\colD{1})=
7\,6\,7\,5\,4\,3\,2\,6\,\colD{1}\,5\,4\,3\,2\,\colD{1},$$
which gives the Dyck word $w_0=\uu\uu\dd\uu\uu\uu\uu\dd\colD{\uu}\dd\dd\dd\dd\colD{\dd}$. In step~2, we obtain $$w_0'=\uu\uu\dd\uu\uu\uu\uu\dd\colD{\uu}\textcolor{green}{\uu\dd}\dd\dd\dd\dd\colD{\dd},$$ where the five $\dd$ steps after the inserted $\textcolor{green}{\uu\dd}$ correspond to $\beta_1\colD{1}=10\,9\,8\,7\,\colD{1}$. After steps~3 and~4, we get
$$w=\uu\uu\dd\uu\uu\uu\uu\dd\colD{\uu}\textcolor{green}{\uu\dd}\colI{\uu\dd}\dd\colI{\uu\dd\uu\dd}\dd\dd\colI{\uu\dd}\dd\colD{\dd}\colU{\uu\dd}.$$

\begin{theorem}\label{thm:1322,2213,2231}
    For all $n\ge 1$, we have $\cc_n(1322,2213,2231) = \Cat_{n+1}-1.$
\end{theorem}

\begin{proof}
    We decompose $\pi \in \cC_n(1322,2213,2231)$ as in Lemma~\ref{lem:decomposition1}. 
    Since $\pi$ avoids $1322$, $\beta_2\gamma$ avoids $211$, and since $\pi$ avoids $2231$, $\alpha\beta_1$ avoids $112$. As in the previous proof, avoidance of $2213$ implies that $\alpha > \gamma$. Let $k=\card{\Set(\alpha)}$.

    If $B_1=\emptyset$, then $\st(\alpha)$ is an arbitrary permutation in $\cC_k(112)$, and $\st(1\beta_21\gamma)$ is an arbitrary permutation in $\cC_{n-k}(112)$. By Theorem~\ref{thm:112}, these sets are enumerated by the Catalan numbers. Summing over $k$, we get 
    \[
    \sum_{k=0}^{n-1} \Cat_k \Cat_{n-k} = \Cat_{n+1}-\Cat_n
    \]
    permutations. 

    Now suppose $B_1\neq\emptyset$. Since $\alpha > \gamma$, avoidance of $1322$ implies that $C=\emptyset$, and that $\beta = \beta_2\beta_1$, that is, the elements of $B_2$ are to the left of those of $B_1$. It follows that 
    \[
    \pi = \alpha\, 1 2 \dots (n-k)\, \beta_1\, 1 2 \dots (n-k),
    \]
    where $\st(\alpha 1 \beta_1 1)$ is an arbitrary permutation in $\cC_{k+1}(112)$ not ending with $11$.
    By Theorem~\ref{thm:112}, these are counted by $\Cat_{k+1}-\Cat_k$. Summing over $k$, we get
    \[
    \sum_{k=1}^{n-1}  \Cat_{k+1} - \Cat_k = \Cat_n - \Cat_1
    \]
    permutations.

    Adding up the permutations in both cases, we have 
    \[
    \cc_n(1322,2213,2231) = (\Cat_{n+1}-\Cat_n) + (\Cat_n - \Cat_1) = \Cat_{n+1} -1.\qedhere
    \]
\end{proof}

\begin{theorem} \label{thm:1132,2213,2311}
For all $n\ge1$, we have $$\cc_n(1132,2213,2311)=\frac{n^3+6n^2-7n+6}{6}.$$
\end{theorem}

\begin{proof}
We have $\cC_n(1132,2213,2311)\subseteq\cC_n(1132,2213,2231)$. This is because avoidance of $2311$ implies avoidance of $2331$ by Lemma~\ref{lem:forward}, which in turn, using that $\pi$ avoids $2213$,  implies avoidance of $2231$ by Lemma~\ref{lem:reverse}. 

We decompose $\pi \in \cC_n(1132,2213,2311)$ as in Lemma~\ref{lem:decomposition1}. 
As in the proof of Theorem~\ref{thm:1132,2213,2231}, avoidance of $1132$ requires $\gamma$ to be weakly increasing, and avoidance of $2213$ implies that $\alpha > \gamma$. Additionally, avoidance of $2311$ now requires $\alpha$ to be weakly decreasing, 

By Lemma~\ref{lem:decomposition1}, elements of $B_1$ form decreasing subsequences in both $\alpha$ and $\beta$, whereas elements of $B_2$ form increasing subsequences in both $\beta$ and $\gamma$.
Additionally, since $\alpha$ and $\gamma$ are weakly monotone, it follows from Lemma~\ref{lem:decomposition1} that
there exist $1\le i\le j\le k\le n$ such that
$B_2=\{2,3,\dots,i\}$, 
$C=\{i+1,i+2,\dots,j\}$, 
$B_1=\{j+1,j+2,\dots,k\}$, and
$A=\{k+1,k+2,\dots,n\}$.
Let us consider three cases depending on the cardinality of $B_1$.

If $|B_1| \geq 2$ (that is, $k-j\ge2$), the property $\alpha > \gamma$, together with avoidance of $2311$, forces $C=\emptyset$ (that is, $i=j$). Avoidance of $2311$ also requires that, in $\beta$, the elements of $B_2$ appear to the left of the elements of $B_1$. Therefore, 
any choice of $j,k$ satisfying $1\leq j < k-1 \leq n-1$ determines the permutation
\begin{equation}\label{eq:piC=0}\pi=nn(n-1)(n-1)\dots(k+1)(k+1)\,k(k-1)\dots(j+1)\,12\dots j\,k(k-1)\dots(j+1)\,12\dots j\end{equation}
This leaves $\binom{n-1}{2}$ permutations in this case. 

If $B_1=\emptyset$ (that is, $j=k$),  
$\pi$ is uniquely determined by the values $i,j$ such that $1\le i\le j\le n$, leaving $\binom{n+1}{2}$ permutations. 

If $|B_1|=1$ (that is, $B_1=\{j+1\}=\{k\}$), there are no restrictions on the position of this entry in $\beta$. Thus, $\pi$ is determined by the values $i,k$ such that $1\le i<k\le n$, and the choice of the position of the entry $k$ in $\beta$, for which we have $|\beta|=i$ choices. This leaves \begin{equation}\label{eq:sumik} \sum_{1\le i<k\le n} i=\binom{n+1}{3}\end{equation} permutations. 

Adding up the three cases, we obtain
$$\cc_n(1132,2213,2311)=\binom{n-1}{2} + \binom{n+1}{2} + \binom{n+1}{3} =\dfrac{n^3+6n^2-7n+6}{6}.\qedhere$$
\end{proof}

We note that Lemmas~\ref{lem:forward} and~\ref{lem:reverse} imply that $\cC_n(1132,2213,2311)=\cC_n(1132,2113,2311)$. 

The next lemma will be useful in some of the upcoming proofs. Since avoidance of $221$ is equivalent to avoidance of $211$, we have  $\cC_n(221,2133)=\cC_n(211,2133)$.

\begin{lemma}\label{lem:221,2133}
    For all $n\ge1$, permutations $\pi\in\cC_n(221,2133)$ are those of the form 
    $$\pi=1122\dots ii\,(i+1)(i+2)\dots n\,(i+1)(i+2)\dots n$$
    for some $0\le i<n$. In particular, $\card{\cC_n(221,2133)}=n$.
\end{lemma}

\begin{proof}
    Viewing $\pi\in\cC_n$ as a nonnesting matching, avoidance of $221$ is equivalent to the labels of the arcs being increasing from left to right, similarly to the proof of Theorem~\ref{thm:112}. Additionally, for any three arcs labeled $a_1<a_2<a_3$ from left to right, if the arcs $a_1$ and $a_2$ cross each other, then the arc $a_3$ must cross both of them; otherwise $\tau$ would contain the subsequence $a_2a_1a_3a_3$, which is an occurrence of $2133$. This forces $\pi$ to have the stated form, and it is clear that such a permutation avoids $2133$.
\end{proof}

\begin{theorem} \label{thm:1332,2133,2311}
    For all $n\ge 1$, we have $\cc_n(1332,2133,2311)=2n^2-3n+2$.
\end{theorem}

\begin{proof}
We decompose $\pi \in \cC_n(1332,2133,2311)$ as in Lemma~\ref{lem:decomposition1}. Avoidance of $2311$ implies that $\alpha$ is weakly decreasing. Avoidance of $2133$ implies avoidance of $2113$ by Lemma~\ref{lem:forward}, which forces $\alpha>\gamma$. 
Since $\pi$ avoids $1332$, $\beta_2\gamma$ must avoid $221$, and since $\beta_2$ has no repeated entries, $\tau\coloneqq1\beta_2 1\gamma$ avoids $221$ as well. By Lemma~\ref{lem:221,2133},
\begin{equation}\label{eq:tau}
\tau = 1122\dots ii\,(i+1)(i+2)\dots j\,(i+1)(i+2)\dots j
\end{equation}
for some $0\le i < j \le n$.

If $|B_1|\ge2$,  avoidance of $2311$ requires that $i=0$, and that, in $\beta$, the elements of $B_2$ appear to the left of those in $B_1$. Thus, we get the same $\binom{n-1}{2}$ permutations as in equation~\eqref{eq:piC=0}.

If  $B_1=\emptyset$, we have $\pi=nn(n-1)(n-1)\dots(j+1)(j+1)\,\tau$, with $\tau$ as in equation~\eqref{eq:tau}, giving $\binom{n+1}{2}$ permutations.

If $|B_1|=1$, we must have $B_1=\{j+1\}$, and $\alpha=nn(n-1)(n-1)\dots(j+2)(j+2)\,(j+1)$. 
For any $1\le i < j \le n-1$ in equation~\eqref{eq:tau}, we can insert the other copy of $j+1$ in between the two $1$s, giving $\binom{n-1}{2}$ permutations.
If $i=0$ in equation~\eqref{eq:tau}, we have $\tau = 12\dots j\,12\dots j$ for some $1\le j\le n-1$, and we can insert the entry $j+1$ in $j$ possible positions, giving $\sum_{j=1}^{n-1} j=\binom{n}{2}$ permutations.

Adding up all the cases, we get 
$$\cc_n(1332,2133,2311) = \binom{n-1}{2}+\binom{n+1}{2}+\binom{n-1}{2}+\binom{n}{2}=2n^2-3n+2.\qedhere$$
\end{proof}

\begin{theorem} \label{thm:1132,2133,2311}
    For all $n\ge2$, we have $\cc_n(1132,2133,2311) = n^2+n-2$.
\end{theorem}

\begin{proof}
We claim that $\cC_n(1132,2133,2311)=\cC_n(1132,2213,2311)\cap\cC_n(1332,2133,2311)$,
the sets from Theorems~\ref{thm:1132,2213,2311} and~\ref{thm:1332,2133,2311}.
For the inclusion to the right, note that avoidance of $1132$ implies avoidance of $1332$ by Lemma~\ref{lem:forward}. By the same lemma, 
avoidance of $2133$ implies avoidance of $2113$, which then implies avoidance of $2213$ by Lemma~\ref{lem:reverse}, using the fact that $\pi$ avoids $2311$. Inclusion to the left is trivial.

We will follow the proof of Theorem~\ref{thm:1132,2213,2311} and count only permutations that avoid $2133$. In the case $|B_1|\ge2$, the $\binom{n-1}{2}$ permutations from equation~\eqref{eq:piC=0}  avoid $2133$.

If $B_1=\emptyset$, avoidance of $2133$ requires that either $i=1$, giving $n$ permutations (one for each $1\le j\le n$), or that $2\le i=j\le n$, giving $n-1$ permutations.

If $|B_1|=1$, avoidance of $2133$ requires that either $i=1$, giving $n-1$ permutations (one for each $1<k\le n$), or that $2\le i=k-1$, giving
$\sum_{i=2}^{n-1} i=\binom{n}{2}-1$ permutations, assuming that $n\ge2$, by changing equation~\eqref{eq:sumik} accordingly.

In total, we have
 $$\cc_n(1132,2133,2311) = \binom{n-1}{2}+n+(n-1)+(n-1)+\binom{n}{2}-1= n^2+n-2.\qedhere$$
\end{proof}

\begin{theorem} \label{thm:1322,2133,2311}
For all $n \ge 1$, we have $\cc_n(1322,2133,2311) = n^2$.
\end{theorem}

\begin{proof}
Let us show that $\cC_n(1322,2133,2311)=\cC_n(1332,2133,2311)\cap\cC_n(1322,2213,2231)$, which are the sets from Theorems~\ref{thm:1332,2133,2311} and~\ref{thm:1322,2213,2231}.
Indeed, avoidance of $1322$ implies avoidance of $1332$ by Lemma~\ref{lem:forward}.
On the other hand, avoidance of $2133$ and $2311$ implies avoidance of $2113$ and $2331$ by  Lemma~\ref{lem:forward}, which imply avoidance of $2213$ and $2231$ by Lemma~\ref{lem:reverse}. Inclusion to the left is straightforward.

Let us follow the proof of Theorem~\ref{thm:1332,2133,2311} and count only permutations that also avoid $1322$. 
In the cases $|B_1|\ge2$ and $B_1=\emptyset$, the same $\binom{n-1}{2}+\binom{n+1}{2}$ permutations avoid $1322$. In the case $|B_1|=1$, avoidance of $1322$ requires $i=0$, and inserting the entry $j+1$ in the rightmost available position (i.e., as the last entry of $\beta$), giving $n-1$ permutations (one for each $1\le j\le n-1$).

In total, we have
 $$\cc_n(1332,2133,2311) = \binom{n-1}{2}+\binom{n+1}{2}+n-1= n^2.\qedhere$$
\end{proof}

The next two theorems consider subsets of $\cC_n(1223,1332,3112)$, the third set in Theorem~\ref{thm:ijjk}(e). Neither of these subsets is contained in the other.

\begin{theorem}\label{thm:1233,1322,3122}
    For all $n\ge 1$, we have $\cc_n(1233,1322,3122) = \Cat_{n+1}-1.$
\end{theorem}

\begin{proof}
    By taking the reverse-complement, we have $\cc_n(1233,1322,3122)=\cc_n(1123,2213,2231)$. We will enumerate permutations $\pi\in\cC_n(1123,2213,2231)$ by decomposing them as in Lemma~\ref{lem:decomposition1}. As in the proof of Theorem~\ref{thm:1132,2213,2231}, avoidance of $2213$ implies that $\alpha>\gamma$, and avoidance of $2231$ implies that $\alpha\beta_1$ avoids $112$. The only difference is that now the third avoided pattern is $1123$ instead of $1132$, so now $\gamma$ has to be weakly decreasing instead of weakly increasing. We can write 
    $\beta_2=j(j-1)\dots (j-i+2)$ and $\gamma=j(j-1)\dots (j-i+2)\,(j-i+1)(j-i+1)(j-i)(j-i)\dots22$ for some $1\le i\le j\le n$. These are the only restrictions on $\alpha$, $\beta_1$, $\beta_2$ and $\gamma$, in the sense that they guarantee that $\pi$ avoids the three patterns $1123,2213,2231$.

    Thus, if we take the complement of $\beta_2\gamma$, by replacing each entry $b\in\{2,3,\dots,j\}$ with $j+2-b$, and keep all the other entries unchanged, the decomposition of the resulting permutation $\pi'$ satisfies precisely the restrictions given in the proof of Theorem~\ref{thm:1132,2213,2231}  when characterizing permutations that avoid $1132$, $2213$ and $2231$. Thus, the map $\pi\mapsto\pi'$ is a bijection from $\cC_n(1123,2213,2231)$ to $\cC_n(1132,2213,2231)$. In particular, by Theorem~\ref{thm:1132,2213,2231}, we have $$\cc_n(1123,2213,2231)=\cc_n(1132,2213,2231)=\Cat_{n+1}-1.\qedhere$$
\end{proof}

\begin{theorem} \label{thm:1123,1132,3312}
For all $n\ge2$, we have $$\cc_n(1123,1132,3312)=\dfrac{7n^2 - 17n + 14}{2}.$$
\end{theorem}

\begin{proof}
Let $\Lambda=\{1123,1132,3312\}$ and let $n\ge3$. We decompose $\pi\in\cC_n(\Lambda)$ as in Lemma~\ref{lem:decomposition1}. Avoidance of $1123$ and $1132$ implies that $\card{\Set(\gamma)}\le 1$. 

Consider the first the case $\card{\Set(\gamma)}=1$. Avoidance of $3312$ implies that $\Set(\gamma)=\{n\}$.
Since $\pi$ avoids $1123$ and ends with $n$, the permutation $\pi'$ obtained from $\pi$ by removing the two copies of $n$ must avoid $112$. Since $\pi$ also avoids $3312$, Lemma~\ref{lem:221,2133}
implies that 
$$\pi'=(n-1)(n-2)\dots(i+1)\,(n-1)(n-2)\dots(i+1)\,ii(i-1)(i-1)\dots11$$
for some $0\le i\le n-2$. If $i=0$, then there are $n$ possible positions for the first copy of $n$, namely immediately before the second copy of $j$ for any $1\le j\le n$, giving $n$ permutations. For each $1\le i\le n-2$, there are two possible positions for the first copy of $n$, namely immediately before or after the second $1$, giving $2(n-2)$ permutations.

Suppose now that $\gamma=\emptyw$. If $A=\emptyset$, then $\pi=\beta1\beta1$, where $\beta$, after subtracting $1$ from each entry, is an arbitrary permutation in $\cS_{n-1}(123,132,312)$. Since there are $n-1$ such permutations~\cite[Prop.~16 and 16$^*$]{simion_schmidt}, this gives $n-1$ possibilities for $\pi$.
If $A\neq\emptyset$, avoidance of $3312$ requires that $A<B_1$. Now avoidance of $1132$ implies that $\card{B_1}\le1$. We consider two cases.

If $B_1=\emptyset$, then $\pi=\alpha11$, where $\st(\alpha)$ is an arbitrary permutation in $\cC_{n-1}(\Lambda)$. Thus, there are $\cc_{n-1}(\Lambda)$ permutations of this form. 

If $|B_1|=1$, the condition $A<B_1$ implies that $B_1=\{n\}$, and so $\beta=n$.
In this case, if $\alpha'$ is the permutation obtained by removing the copy of $n$ from $\alpha$, then 
$\st(\alpha')$ is an arbitrary permutation in $\cC_{n-2}(112,3312)$. Indeed, $\alpha'$ must avoid $112$ because $\pi$ avoids $1123$, and one can check that if $\alpha'$ avoids $112$ and $3312$, then $\pi$ avoids the three patterns in $\Lambda$. By Lemma~\ref{lem:221,2133},
$$\alpha'=(n-1)(n-2)\dots(i+1)\,(n-1)(n-2)\dots(i+1)\,ii(i-1)(i-1)\dots22$$
for some $1\le i\le n-2$.
If $i=1$, then the first copy of $n$ can be inserted in $n-1$ positions in $\alpha'$, namely immediately before of the second copy of $j$ for any $2\le j\le n-1$, or at the end.
If $2\le i\le n-2$, then the first copy of $n$ can be inserted two positions, namely immediately before or after the second $2$, giving $2(n-3)$ permutations.

Combining all the cases, we get the recurrence
$$\cc_n(\Lambda)=n+2(n-2)+(n-1)+\cc_{n-1}(\Lambda)+(n-1)+2(n-3)=
\cc_{n-1}(\Lambda)+7n-12.$$
Using the initial condition $\cc_2(\Lambda)=4$, we deduce the stated formula for $\cc_n(\Lambda)$.
\end{proof}

In the next seven theorems, we consider subsets of $\cC_n(1223,2331,3112)$, the fourth set in Theorem~\ref{thm:ijjk}(e).
Figure~\ref{fig:1223,2331,3112} shows the containment relationships between these sets.

\begin{figure}[ht]
    \centering
    \begin{tikzpicture}[scale=1.1]
        \node (A) at (0,3) {\fbox{$\cC_n(1223,2331,3112)$} \ref{thm:ijjk}(e)};
        \node (B) at (3.5,2) {\fbox{$\cC_n(1223,2231,3112)$} \ref{thm:1223,2231,3112}};
        \node (C) at (-4.5,1) {\fbox{$\cC_n(1123,2311,3112)$} \ref{thm:1123,2311,3112}};
        \node (D) at (0,1) {\fbox{$\cC_n(1123,2331,3312)$} \ref{thm:1123,2331,3312}};
        \node (E) at (4.5,1) {\fbox{$\cC_n(1223,2231,3312)$} \ref{thm:1223,2231,3312}};
        \node (F) at (-4.5,0) {\fbox{$\cC_n(1123,2311,3122)$} \ref{thm:1123,2311,3122}};
        \node (G) at (0,0) {\fbox{$\cC_n(1123,2311,3312)$} \ref{thm:1123,2311,3312}};
        \node (H) at (4.5,0) {\fbox{$\cC_n(1123,2231,3312)$} \ref{thm:1123,2231,3312}};
        \draw (A)--(B)--(E)--(H)--(D)--(A)--(C)--(G)--(D); \draw(C)--(F);
    \end{tikzpicture}
    \caption{The subsets of $\cC_n(1223,2331,3112)$ enumerated in this section.}
    \label{fig:1223,2331,3112}
\end{figure}
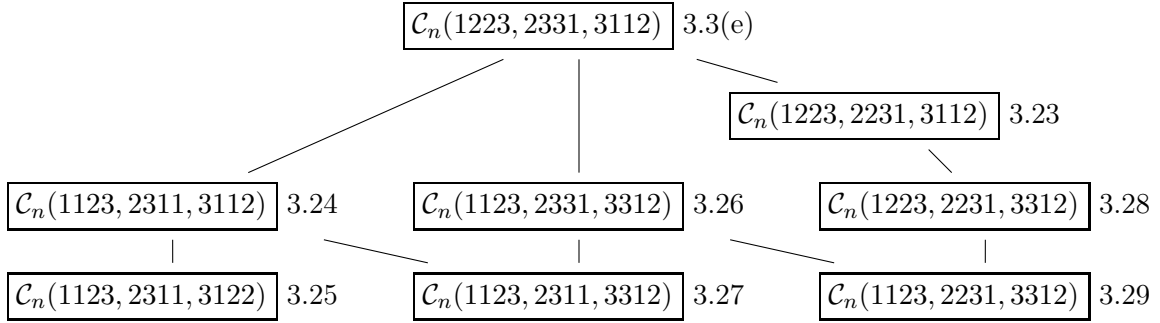
\begin{theorem}\label{thm:1223,2231,3112}
    For all $n\ge 1$, we have $\cc_n(1223,2231,3112) = \Cat_{n+1}-1$.
\end{theorem}

\begin{proof}
    We decompose $\pi \in \cC_n(1223,2231,3112)$ as in Lemma~\ref{lem:decomposition1}. Avoidance of $3112$ implies that $\alpha<\gamma$. Let $k\in[n]$ be such that $\Set(\alpha)=\{2,3,\dots,k\}$ and $\Set(\gamma)=\{k+1,k+2,\dots,n\}$. Since $\pi$ avoids $2231$, $\alpha\beta_1$ must avoid $112$, and since $\pi$ avoids $1223$, $\beta_2\gamma$ must avoid $112$ as well. Thus, the underlying permutations of $\alpha\beta_1$ and $\beta_2\gamma$ are decreasing, which implies that $\beta_1$ and $\beta_2$ are decreasing, since $\beta_1$ consists of only right copies of entries, and $\beta_2$ consists of only left copies. Additionally, avoidance of $2231$ forces $\beta_2$ to be to the left of $\beta_1$; otherwise, if $b_1\in B_1$ appears to the left of $b_2\in B_2$ within $\beta$, the subsequence $b_1b_1b_21$ (where the first copy of $b_1$ is in $\alpha$) would be an occurrence of $2231$. 

    If $\beta_2=\emptyw$, then $\alpha1\beta1$ is an arbitrary permutation in $\cC_k(112)$, and $\st(\gamma)$ is an arbitrary permutation in $\cC_{n-k}(112)$. Thus, by Theorem~\ref{thm:112}, there are $$\sum_{k=1}^{n}\Cat_k\Cat_{n-k}=\Cat_{n+1}-\Cat_n$$ possibilities for $\pi$ in this case.

    If $\beta_2\neq\emptyw$, then avoidance of $2231$, together with the fact that $\alpha<\beta$, forces $A=\emptyset$. In this case, we have
    $$\pi=k(k-1)\dots 1\,\beta_2\,k(k-1)\dots 1\,\gamma,$$
    where $\st(1\beta_21\gamma)$ is an arbitrary permutation in $\cC_{n-k+1}$ whose underlying permutation is $1(k+1)k\dots 2$ and does not start with $11$. Indeed, $\st(1\beta_21\gamma)$ has these properties because $\beta_2\gamma$ has a decreasing underlying permutation, and $\beta_2\neq\emptyw$. Additionally, these properties guarantee that $\pi$ avoids the patterns $1223,2231,3112$. Since the number of permutations in $\cC_{n-k+1}$ with a given underlying permutation is $\Cat_{n-k+1}$ and the number of those that start with $11$ is $\Cat_{n-k}$, the total number of possibilities for $\pi$ in this case is
    $$\sum_{k=1}^{n-1}(\Cat_{n-k+1}-\Cat_{n-k})=\Cat_n-\Cat_1.$$

    Adding up both cases, we obtain
    $$\cc_n(1223,2231,3112) = (\Cat_{n+1}-\Cat_n)+(\Cat_n-\Cat_1)=\Cat_{n+1}-1.\qedhere$$
\end{proof}

\begin{theorem} \label{thm:1123,2311,3112}
    For all $n\ge2$, we have $$\cc_n(1123,2311,3112)=\dfrac{n^3+3n^2+8n-12}{6}.$$
\end{theorem}

\begin{proof}
Let $n\ge2$, and decompose $\pi \in \cC_n(1123,2311,3112)$ as in Lemma~\ref{lem:decomposition1}. Avoidance of $1123$ and $2311$ forces $\gamma$ and $\alpha$ to be weakly decreasing, respectively.
Avoidance of $3112$ requires $\alpha < 
\gamma$.

If $B_2=\emptyset$, we must have
\begin{equation}\label{eq:B2empty}\pi=jj(j-1)(j-1)\dots (i+1)(i+1)\, i(i-1)\dots 1\, i(i-1) \dots 1\, nn (n-1)(n-1) \dots (j+1)(j+1)\end{equation}
for some $1\le i \le j \le n$, giving $\binom{n+1}{2}$ permutations.

If $|B_2|=1$, Lemma~\ref{lem:decomposition1}, along with the fact that $\alpha<\gamma$ and $\gamma$ is weakly decreasing, imply that $B_2=\{n\}$.
In this case, $\pi$ has a form similar to equation~\eqref{eq:B2empty}, with $1\le i \le j \le n-1$,  but where the first $n$ is instead inserted in $\beta$ (i.e., between the two copies of $1$), in one of the $i$ available positions.
The number of permutations of this form is
\begin{equation}\label{eq:sumij}\sum_{1\le i\le j\le n-1} i=\binom{n+1}{3}.
\end{equation}

Finally, consider the case $|B_2|\ge2$. Avoidance of $2311$ forces $C=\emptyset$, and avoidance of $1123$ forces $A=\emptyset$. Therefore,
\begin{equation}\label{eq:3stripes}\pi=i(i-1)\dots 1\,n(n-1)\dots 1\,n(n-1)\dots(i+1)\end{equation}
for some $1\le i\le n-2$, giving $n-2$ permutations.

Adding up the three cases,
$$\cc_n(1123,2311,3112) = \binom{n+1}{2} + \binom{n+1}{3}+ n-2 = \dfrac{n^3+3n^2+8n-12}{6}.\qedhere$$
\end{proof}

\begin{theorem} \label{thm:1123,2311,3122}
    For all $n\ge2$, we have $\cc_n(1123,2311,3122) = n^2+n-2$.
\end{theorem}

\begin{proof}
Since $\pi$ avoids $3122$, it also avoids $3112$ by Lemma~\ref{lem:forward}. It follows that
$\cC_n(1123,2311,3122)\subseteq \cC_n(1123,2311,3112)$, the set that we enumerated in Theorem~\ref{thm:1123,2311,3112}.
In the proof of this theorem, the only case where $\pi$ may contain the pattern $3122$ is when $B_2=\{n\}$. In this case, we must have $j=n-1$ in order to avoid $3122$. Therefore, equation~\eqref{eq:sumij} becomes
$$
\sum_{i=1}^{n-1} i=\binom{n}{2},
$$
and adding up the three cases, we now get
$$\cc_n(1123,2311,3122) = \binom{n+1}{2} + \binom{n}{2}+ n-2 = n^2+n-2.\qedhere$$
\end{proof}

\begin{theorem} \label{thm:1123,2331,3312}
    For all $n\ge1$, we have $\cc_n(1123,2331,3312) = 2n^2-3n+2$.
\end{theorem}

\begin{proof}
We decompose $\pi \in \cC_n(1123,2331,3312)$ as in Lemma~\ref{lem:decomposition1}. Avoidance of $1123$ forces $\gamma$ to be weakly decreasing. By Lemma~\ref{lem:forward}, avoidance of $3312$ implies avoidance of $3112$, which requires $\alpha < \gamma$.
 
Since $\pi$ avoids $2331$, $\alpha\beta_1$ avoids $122$, and so does $\tau\coloneqq\alpha1\beta_1 1$. Since $\tau$ also avoids $3312$, Lemma~\ref{lem:221,2133} applied to the reversal of $\tau$ implies that 
\begin{equation}\label{eq:tau-decreasing}\tau =j(j-1)\dots (i+1)\,j(j-1)\dots (i+1)\,ii(i-1)(i-1)\dots11\end{equation}
for some $0\le i<j\le n$.

If $B_2=\emptyset$, then $\gamma=nn(n-1)(n-1)\dots (j+1)(j+1)$, and any choice of 
$0\le i<j\le n$ gives a valid $\pi$, producing $\binom{n+1}{2}$ permutations.

If $|B_2|=1$, then $B_2=\{n\}$ and $\gamma=n\,(n-1)(n-1)\dots (j+1)(j+1)$. If $i=0$, then $\tau =j(j-1)\dots 1\,j(j-1)\dots 1$, and we can insert the other $n$ in any of the $j$ positions in $\beta$, giving $\sum_{j=1}^{n-1}j=\binom{n}{2}$ possibilities for $\pi$.
If $i\ge1$ (that is, $B_1=\emptyset$), then the other $n$ is the only entry in $\beta$, so we get 
$\binom{n-1}{2}$ possibilities, one for each choice of $1\le i<j\le n-1$.

Finally, if $|B_2|\ge2$, avoidance of $1123$ implies that $A=\emptyset$ and that, in $\beta$, the elements of $B_2$ are to the left of those of $B_1$. It follows that
$$\pi=j(j-1)\dots 1\,n(n-1)\dots (k+1)\, j(j-1)\dots 1\,n(n-1)\dots (k+1)\, kk(k-1)(k-1)\dots(j+1)(j+1)$$
for some $1\le j\le k\le n-2$, giving $\binom{n-1}{2}$ permutations.

Adding up all the cases,
$$\cc_n(1123,2331,3312) = \binom{n+1}{2} + \binom{n}{2} + \binom{n-1}{2} + \binom{n-1}{2}= 2n^2-3n+2.\qedhere$$
\end{proof}

\begin{theorem} \label{thm:1123,2311,3312}
    For all $n\ge2$, we have $$\cc_n(1123,2311,3312) = \dfrac{n^2+7n-10}{2}.$$
\end{theorem}

\begin{proof}
We have that $\cC_n(1123,2311,3312)=\cC_n(1123,2311,3112)\cap\cC_n(1123,2331,3312)$,
the sets from Theorems~\ref{thm:1123,2311,3112} and~\ref{thm:1123,2331,3312}.
Indeed, by Lemma~\ref{lem:forward}, avoidance of $3312$ implies avoidance of $3112$, and avoidance of $2311$ implies avoidance of $2331$. Inclusion to the left is trivial.

Let us follow the proof of Theorem~\ref{thm:1123,2311,3112} and count only permutations that avoid $3312$.
 If $B_2=\emptyset$, the permutation in equation~\eqref{eq:B2empty} avoids $3312$ only if $1=i\le j\le n$, giving $n$ permutations, or if $2\le i=j\le n$, giving $n-1$ permutations.
 If $B_2=\{n\}$, the copy of $n$ in $\beta$ can be inserted in one position if $1=i\le j\le n-1$, giving $n-1$ permutations, and in $i$ positions if $2\le i=j\le n$, giving $\sum_{i=2}^{n-1}i=\binom{n}{2}-1$ permutations.
 If $|B_2|\ge2$, all $n-2$ permutation in equation~\eqref{eq:3stripes} avoid $3312$. 

Adding up the three cases,
$$\cc_n(1123,2311,3312) = n+(n-1)+(n-1)+\binom{n}{2}-1+ (n-2) = \dfrac{n^2+7n-10}{2}.\qedhere$$
\end{proof}

\begin{theorem} \label{thm:1223,2231,3312}
    For all $n\ge1$, we have $$\cc_n(1223,2231,3312) = \dfrac{n^3+2n}{3}.$$
\end{theorem}

\begin{proof}
We have $\cC_n(1223,2231,3312)\subseteq\cC_n(1223,2231,3112)$ by Lemma~\ref{lem:forward}. As in the proof of Theorem~\ref{thm:1223,2231,3112}, decomposing $\pi \in \cC_n(1223,2231,3312)$ as in Lemma~\ref{lem:decomposition1}, we have $\alpha < \gamma$, and the elements of $\beta_2$ are to the left of those of $\beta_1$. Thus, we can write $\pi=\alpha1\beta_2\beta_11\gamma$.

Since $\pi$ avoids $2231$, $\alpha\beta_1$ avoids $112$, and so the word $\tau_1\coloneqq\alpha1\beta_11$ is a nonnesting permutation that avoids $112$ and $3312$. By Lemma~\ref{lem:221,2133} applied to the reversal of $\tau_1$, we have
$$\tau_1 =j(j-1)\dots (i+1)\,j(j-1)\dots (i+1)\,ii(i-1)(i-1)\dots11$$
for some $0\le i<j\le n$.
On the other hand, the word $\tau_2\coloneqq\beta_2\gamma$, after standardizing (by subtracting $j$ from each entry), is also a nonnesting permutation that avoids $112$ (since $\pi$ avoids $1223$) and $3312$, so again we must have 
$$\tau_2 =n(n-1)\dots (k+1)\,n(n-1)\dots (k+1)\,kk(k-1)(k-1)\dots(j+1)(j+1)$$
for some $j\le k\le n$.

Now let us analyze how $\tau_1$ and $\tau_2$ can overlap with each other. 
If $i=0$, the above conditions imply that
\begin{multline*}\pi=j(j-1)\dots 1\,n(n-1)\dots(\ell+1)\,j(j-1)\dots 1\\
\,\ell(\ell-1)\dots (k+1)\,n(n-1)\dots(k+1)\,kk(k-1)(k-1)\dots (j+1)(j+1),\end{multline*}
for some $1\le j\le k\le \ell\le n$. However, to avoid double-counting, we do not count the case when $k=n-1$ and $\ell=n$, since, for any given $j\in[n-1]$, such indices would produce the same permutation $\pi$ as when $k=\ell=n$. This gives $\binom{n+2}{3}-(n-1)$ different permutations.

If $i\ge1$, then avoidance of $2231$ forces $\pi=\tau_1\tau_2$. In this case, we get a permutation for each choice of indices $1\le i<j\le k\le n$, but again, to avoid double-counting, we do not allow $k=n-1$. This gives
$\binom{n+1}{3}-\binom{n-1}{2}$ different permutations.

Adding the two cases,
$$\cc_n(1223,2231,3312) = \binom{n+2}{3}-(n-1)+ \binom{n+1}{3}-\binom{n-1}{2}= \dfrac{n^3+2n}{3}.\qedhere$$
\end{proof}

\begin{theorem} \label{thm:1123,2231,3312}
    For all $n\ge 1$, we have $\cc_n(1123,2231,3312) = n^2.$
\end{theorem}

\begin{proof}
We have that $\cC_n(1123,2231,3312)=\cC_n(1123,2331,3312)\cap\cC_n(1223,2231,3312)$,
the sets from Theorems~\ref{thm:1123,2331,3312} and~\ref{thm:1223,2231,3312}.
This is because, by Lemma~\ref{lem:forward}, avoidance of $2231$ implies avoidance of $2331$, and avoidance of $1123$ implies avoidance of $1223$. 

Let us follow the proof of Theorem~\ref{thm:1123,2331,3312} and consider only permutations that avoid $2231$.
In the cases $B_2=\emptyset$ and $|B_2|\ge2$, all $\binom{n+1}{2}+\binom{n-1}{2}$ permutations in that proof avoid $2231$.

In the case $B_2=\{n\}$, if $i=0$, then $n$ has to be inserted in the first position of $\beta$ in order for $\pi$ to avoid $2231$, giving $n-1$ permutations coming from the choices of $1\le j\le n-1$ in equation~\eqref{eq:tau-decreasing}.
If $i\ge1$, all permutations contain $2231$, so we do not count them here.

Adding up all the cases,
$$\cc_n(1123,2231,3312) = \binom{n+1}{2} +\binom{n-1}{2} + n-1= n^2.\qedhere$$
\end{proof}

The last result in this subsection concerns a set of nonnesting permutations avoiding two patterns. Despite the simple formula, the proof is a more technical than the above ones.
We define a grand Dyck word of length $2n$ to be a sequence of $n$ $\uu$s and $n$ $\dd$s with no other restrictions. It is well know that the the number of grand Dyck words of length $2n$ is $\binom{2n}{n}$, and that its generating function is $\sum_{n\ge0}\binom{2n}{n}=\frac{1}{\sqrt{1-4x}}$.

\begin{theorem} \label{thm:1322,2231}
For all $n\ge1$, we have $$\cc_n(1322,2231)=\binom{2n}{n}-2^{n-1}.$$ 
\end{theorem}

\begin{proof}
We decompose $\pi\in\cC_n(1322,2231)$ as in Lemma~\ref{lem:decomposition1}.
Since $\pi$ avoids $1322$, $\beta_2\gamma$ avoids $211$, so its underlying permutation is increasing. Similarly, since $\pi$ avoids $2231$, $\alpha\beta_1$ avoids $112$, so its underlying permutation is decreasing. It follows that $\beta_2$ is increasing and $\beta_1$ is decreasing.

Let us show that $\beta=\beta_2\beta_1$, that is the elements in $B_2$ are to the left of the elements of $B_1$ in $\beta$. Suppose for contradiction that $b_1\in B_1$ and $b_2\in B_2$ and that $b_1$ is to the left of $b_2$ within $\beta$. Then $\beta$ contains the subsequence $b_11b_1b_21b_2$. If $b_1<b_2$, then $b_1b_1b_21$ is an occurrence of $2231$, and if $b_1>b_2$, then $1b_1b_2b_2$ is an occurrence of $1322$, which is a contradiction in both cases.

If $A\cup C=\emptyset$, we have $\pi=\beta_1 1\beta_2\beta_11\beta_2$, and $\pi$ is determined by which elements from $\{2,3,\dots,n\}$ are in $B_1$, giving $2^{n-1}$ permutations if $n\ge1$. The corresponding ordinary generating function is 
\begin{equation}\label{eq:ACempty} 1+\sum_{n\ge1}2^{n-1}x^n=\frac{1-x}{1-2x}.\end{equation}

Suppose now that $A\cup C\neq\emptyset$, and let us show that $B<A\cup C$. Indeed, if there were elements $b\in B$ and $a\in A$ such that $a<b$, then $aab1$ would an occurrence of $2231$. Similarly, if there was a $c\in C$ such that $c<b$, then $1bcc$ would an occurrence of $1322$. Let us assume that the smallest element in $A\cup C$ is in $A$; the case where it is in $C$ is symmetric.

Since $\pi$ avoids $1322$, for each $c\in C$, all the elements of $\alpha$ larger than $c$ must come before those smaller than $c$. Similarly, since $\pi$ avoids $2231$, for each $a\in A$, all the elements of $\gamma$ smaller than $a$ must come before those larger than $a$. This means that we have disjoint unions $A=A_1\sqcup A_2\sqcup \dots$ and $C=C_1\sqcup C_2\sqcup \dots$ where the $A_i$ and $C_i$ are nonempty intervals (i.e., sets of consecutive integers) such that
\begin{equation}\label{eq:orderAC}
B<A_1<C_1<A_2<C_2<\dots,
\end{equation}
all the elements of $A_{i+1}$ appear to the left of those of $A_i$, and all the elements of $C_i$ appear to the left of those of $C_{i+1}$, for all $i$.

Denoting the restriction of $\pi$ to the elements in each $A_i$ by $\pi|_{A_i}$, the word $\st(\pi|_{A_i})$ is an arbitrary nonnesting permutation avoiding $112$, so it can be encoded as a Dyck word by replacing the first copy of each entry with a $\uu$ and the second copy with a $\dd$. Denote this word by $w(\pi|_{A_i})$. Similarly, $\st(\pi|_{C_i})$ forms an arbitrary nonnesting 
permutation avoiding $211$, which can be encoded as a Dyck word $w(\pi|_{C_i})$.
It follows that the restriction of $\pi$ to $A\cup C$ can be encoded as a grand Dyck word
\begin{equation}\label{eq:grandDyck}
\dots w(\pi|_{C_2})^r\, w(\pi|_{A_2})\, w(\pi|_{C_1})^r\, w(\pi|_{A_1}),
\end{equation}
that is, for each of the sets $A_i$ and $C_i$ in the opposite order from equation~\eqref{eq:orderAC} and we consider their associated Dyck words, reversing the ones coming from the sets $C_i$. Viewing words as lattice paths with steps $\uu=(1,1)$ and $\dd=(1,-1)$ starting at the origin, the reversed Dyck words correspond to portions of the path below the $x$-axis.

If $\beta=\emptyw$, we can recover $\pi$ uniquely from the above grand Dyck word, which has length $2(n-1)$. However, to deal with arbitrary $\beta$, we have to modify the last portion $w(\pi|_{A_1})$ of the above word, to take into account how the elements of $A_1$ and $B_1$ may be interleaved.
Let $a=\min A_1$, which is the rightmost entry of $\pi|_{A_1}$, and write $B_1=B_1^L\sqcup B_1^R$, where $B_1^L$ (resp.\ $B_1^R$) are the elements whose first copy appears before (resp.\ after) the second copy of $a$. Note that $B_1^L>B_1^R$, since $\beta_1$ is decreasing, and that $B_1^L>B_2$, since otherwise $\pi$ would contain $1322$ (with $a$ playing the role of $3$). Therefore, $B_1^R\cup B_2=\{2,3,\dots,k+1\}$ for some $0\le k\le n-2$. Note also that, in $\gamma$, the elements of $B_2$ must appear to the left of the elements of $C$, since otherwise, any $c\in C$ to the left of $b_2\in B_2$ would create a subsequence $aacb_2$, which is an occurrence of $2231$.

The restriction of $\pi$ to $A_1\cup B_1^L\cup\{1\}$, after standardizing, is a nonnesting permutation avoiding $112$, which can be encoded as a Dyck word $w(\pi|_{A_1\cup B_1^L\cup\{1\}})$. However, this is not an arbitrary Dyck word, but rather one with the property that the rightmost $\uu$ (which corresponds to the first copy of $1$) is preceded by a $\dd$ (which corresponds to the second copy of $a$). Given a Dyck word with this property, we not only can recover $\st(\pi|_{A_1\cup B_1^L\cup\{1\}})$, by we also know that 
$|B_1^L|+1$ is precisely the number of $\dd$s after the last $\uu$.
Dyck words whose rightmost $\uu$ is preceded by a $\dd$, by turning this pair $\dd\uu$ into $\uu\dd$, are in bijection with Dyck words ending in $\dd\dd$. Let $w'(\pi|_{A_1\cup B_1^L\cup\{1\}})$ be the Dyck word obtained after this transformation.
Replacing $w(\pi|_{A_1})$ with $w'(\pi|_{A_1\cup B_1^L\cup\{1\}})$ on the right of equation~\eqref{eq:grandDyck}, we obtain a grand Dyck word ending with $\dd\dd$. Denote this word by $g(\pi)$.

The remaining piece of information needed to determine $\pi$ is which elements of $\{2,3,\dots,k+1\}$ belong to $B_2$, since the rest must belong to $B_1^R$. 

The map $\pi\mapsto (g(\pi),B_2)$ is a bijection from permutations $\pi\in\cC_n(1322,2231)$ where the smallest element of $A\cup C$ is in $A$, to pairs consisting of a grand Dyck word $g(\pi)$ of semilength $n-k$ (for some $0\le k\le n-2$) ending with $\dd\dd$, and a subset $B_2\subseteq\{2,3,\dots,k+1\}$.
For permutations $\pi\in\cC_n(1322,2231)$ where the smallest element of $A\cup C$ is in $C$, a symmetric construction produces a pair $(g(\pi),B_1)$, where the Dyck word
$g(\pi)$ ends with $\uu\uu$, and $B_1\subseteq\{2,3,\dots,k+1\}$.

The generating function for grand Dyck words ending with $\dd\dd$ or $\uu\uu$, or equivalently, nonempty grand Dyck words not ending with $\uu\dd$ or $\dd\uu$, is 
$$\frac{1}{\sqrt{1-4x}}-1-\frac{2x}{\sqrt{1-4x}}.$$
On the other hand, the generating function for subsets of $\{2,3,\dots,k+1\}$ is $\frac{1}{1-2x}$.

Multiplying these and adding equation~\eqref{eq:ACempty}, we deduce that 
$$
\sum_{n\ge0}\cc_n(1322,2231)\,x^n=\left(\frac{1-2x}{\sqrt{1-4x}}-1\right)\frac{1}{1-2x}+\frac{1-x}{1-2x}=\frac{1}{\sqrt{1-4x}}-\frac{x}{1-2x},
$$
and so, extracting coefficients, $$\cc_n(1322,2231)=\binom{2n}{n}-2^{n-1}.\qedhere$$
\end{proof}

As an example of the construction in the proof of Theorem~\ref{thm:1322,2231}, let 
$$\pi=\textcolor{darkgray}{17\,16\,17\,16\,15\,15}\ \colU{12\,11\,10\,12}\,\colI{9\,8}\,\colU{11}\,\colI{7}\,\colU{10}\ \colII{5\,4}\,\colD{1}\,\textcolor{green}{2\,3\,6}\,\colI{9\,8\,7}\,\colII{5\,4}\,\colD{1}\,\textcolor{green}{2\,3\,6}\ \textcolor{brown}{13\,14\,13\,14}\in\cC_{17}(1322,2231),$$
which has $A_2=\{\textcolor{darkgray}{15},\textcolor{darkgray}{16},\textcolor{darkgray}{17}\}$, $C_1=\{\textcolor{brown}{13},\textcolor{brown}{14}\}$, $A_1=\{\colU{10},\colU{11},\colU{12}\}$, $B_1^L=\{\colI{7},\colI{8},\colI{9}\}$, $B_1^R=\{\colII{4},\colII{5}\}$, $B_2=\{\textcolor{green}{2},\textcolor{green}{3},\textcolor{green}{6}\}$, and $k=5$.
Then $w(\pi|_{A_2})=\textcolor{darkgray}{\uu\uu\dd\dd\uu\dd}$, $w(\pi|_{C_1})^r=\textcolor{brown}{\dd\dd\uu\uu}$, 
$$w(\pi|_{A_1\cup B_1^L\cup\{1\}})=w(\colU{12\,11\,10\,12}\,\colI{9\,8}\,\colU{11}\,\colI{7}\,\colU{10}\,\colD{1}\,\colI{9\,8\,7}\,\colD{1})=\colU{\uu\uu\uu\dd}\colI{\uu\uu}\colU{\dd}\colI{\uu}\colU{\dd}\colD{\uu}\colI{\dd\dd\dd}\colD{\dd},$$
and $w'(\pi|_{A_1\cup B_1^L\cup\{1\}})=\colU{\uu\uu\uu\dd}\colI{\uu\uu}\colU{\dd}\colI{\uu}\colD{\uu}\colU{\dd}\colI{\dd\dd\dd}\colD{\dd}$.
Concatenating these words, we get
$$g(\pi)=\textcolor{darkgray}{\uu\uu\dd\dd\uu\dd}\,
\textcolor{brown}{\dd\dd\uu\uu}\,
\colU{\uu\uu\uu\dd}\colI{\uu\uu}\colU{\dd}\colI{\uu}\colD{\uu}\colU{\dd}\colI{\dd\dd\dd}\colD{\dd}.$$

\subsection{Patterns with non-adjacent repeated letters}

In this subsection we consider sets of patterns of length~$4$ that include patterns with repeated letters in non-adjacent positions. For the sets we consider, the number of nonnesting permutations avoiding them is still given by nice formulas. 

\begin{theorem} \label{thm:1132,3112,3121}
For all $n\ge2$, we have $\cc_n(1132,3112,3121)=5 \cdot 3^{n-2}-1$.
\end{theorem}

\begin{proof}
Let $\Lambda=\{1132,3112,3121\}$, and decompose $\pi \in \cC_n(\Lambda)$ as in Lemma~\ref{lem:decomposition1}.
Avoidance of $1132$ requires $\gamma$ to be weakly increasing, avoidance of $3112$ requires $\alpha<\gamma$, and avoidance of $3121$ requires $\alpha\le\beta$, so in particular $|B_1|\le1$.
By Lemma~\ref{lem:decomposition1}, the entries in $B_2$ form an increasing sequence in both $\beta$ and $\gamma$, and $\gamma$ consists of the elements of $B_2$ followed by the elements of $C$, each of which is duplicated.
In particular, $\pi$ ends with $n$ if and only if $\gamma\neq\emptyw$.
Let $\cc_n=\cc_n(\Lambda)$, and let $r_n$ denote the number of permutations in $\cC_n(\Lambda)$ that end with an $n$, so that $\cc_n-r_n$ is the number of those that do not. 

Let us first focus on permutations that do not end with an $n$, namely those with $\gamma=\emptyw$, and suppose that $n\ge2$. If $\beta=\emptyw$, these are permutations of the form $\pi=\alpha11$, where 
$\st(\alpha)$ is an arbitrary element of $\cC_{n-1}(\Lambda)$. If $\beta\neq\emptyw$, then the condition $\alpha\le\beta$ and the nonnesting property implies that $\pi=\alpha1n1$, where $\st(\alpha n)$ is an arbitrary element of $\cC_{n-1}(\Lambda)$ ending with its largest entry. It follows that 
\begin{equation}\label{eq:s-f} \cc_n-r_n=\cc_{n-1}+r_{n-1}. \end{equation}

Now consider permutations in $\cC_n(\Lambda)$ that end with an $n$, and suppose that $n\ge3$. If $C\neq\emptyset$, these permutations end in fact with $nn$, and removing this pair of entries yields an arbitrary permutation in $\cC_{n-1}(\Lambda)$, so these are counted by $\cc_{n-1}$. Suppose now that $C=\emptyset$, which requires $B_2\neq\emptyset$ for the permutation to end with $n$. Consider two cases depending on the cardinality of $A$, with subcases depending on whether $|B_1|$ equals $0$ or $1$.

\begin{itemize} \item Case $A\neq\emptyset$. We must have $|B_2|\leq 1$ in this case; otherwise, taking $i, i' \in B_2$ with $i<i'$, the subsequence $22i'i$, where $i'$ is an entry in $\beta$ and $i$ is an entry in $\gamma$,
would be an occurrence of $1132$. Combined with the above conditions, this forces $B_2=\{n\}$.

If $|B_1|=0$, we must have $\pi=\alpha1n1n$, where $\st(\alpha)$ is an arbitrary element of $\cC_{n-2}(\Lambda)$, giving $\cc_{n-2}$ permutations. 

If $|B_1|=1$, we must have $B_1=\{n-1\}$, where $n-1$ is also the largest entry in $\alpha$, and so
$\pi=\alpha1(n-1)n1n$,
where $\st(\alpha(n-1))$ is an arbitrary element of $\cC_{n-2}(\Lambda)$ ending with its largest entry.
Indeed, if this word avoids $1132$, $3112$ and $3121$, then so does $\pi$. 
This subcase contributes $r_{n-2}$ permutations, except when $n=3$, in which case the resulting permutation $\pi=212313$ has $A=\emptyset$, so it will be counted in the next case instead.

\item Case $A=\emptyset$. If $|B_1|=0$, the only possibility is $\pi=12\dots\,n123\dots n$.

If $|B_1|=1$, we must have $B_1=\{2\}$ and $B_2=\{3,4,\dots,n\}$. If $n\ge4$, avoidance of $1132$ requires that $2$ appears at the end of $\beta$; otherwise $\pi$ would contain the subsequence $22n3$, where $n$ is an entry in $\beta$ and $3$ is an entry in $\gamma$. This forces $\pi=21\,34\dots n\,21\,34\dots n$, except when $n=3$, where we get the additional permutation $\pi=2\,1\,23\,1\,3$ which was not counted in the previous case.
\end{itemize}

Combining both cases for permutations ending with an $n$, we obtain
\begin{equation}\label{eq:f} r_n=\cc_{n-1}+\cc_{n-2}+r_{n-2}+2 \end{equation}
for $n\ge3$.

Adding equations~\eqref{eq:s-f} and~\eqref{eq:f}, and then using the equality $r_{n-1}+r_{n-2}=\cc_{n-1}-\cc_{n-2}$, which follows from equation~\eqref{eq:s-f} with the index shifted by one, we obtain the recurrence
$$\cc_n=2\cc_{n-1}+r_{n-1}+\cc_{n-2}+r_{n-2}+2=3\cc_{n-1}+2$$
for $n\ge3$.
From this recurrence, along with the initial condition $\cc_2=4$, 
we can prove the formula $\cc_n=5 \cdot 3^{n-2}-1$ by induction.
\end{proof}

\begin{theorem} \label{thm:1231,1321,2132,2312,3123,3213}
For all $n\ge1$, we have $\cc_n(1231,1321,2132,2312,3123,3213)=n!\Fib_n$. 
\end{theorem}

\begin{proof}
Let $\Lambda=\{1231,1321,2132,2312,3123,3213\}$. 
Patterns in $\Lambda$ are precisely those of the form $ijki$, where $ijk\in\cS_3$. 
It follows that $\pi\in\cC_n(\Lambda)$ if and only every arc in its associated matching connects adjacent entries, or entries having only one entry in between.

Let $a_n$ be the number of such matchings of $[2n]$. Such a matching either has an arc $(2n-1,2n)$, giving rise to a matching of $[2n-2]$ on the remaining vertices, or it has arcs $(2n-3,2n-1)$ and $(2n-2,2n)$, giving rise to a matching of $[2n-4]$. Therefore, $a_n=a_{n-1}+a_{n-2}$, with initial conditions $a_1=1$ and $a_2=2$, implying that $a_n=\Fib_n$. 

Each matching can be labeled in $n!$ ways to form a permutation in $\cC_n(\Lambda)$, proving the stated formula.
\end{proof}

Our last two results are proved using exponential and ordinary generating functions, respectively.

\begin{theorem} \label{thm:1231,1321}
The exponential generating function for nonnesting permutations that avoid $\{1231,1321\}$ is
\[
\sum_{n\ge0} \cc_n(1231,1321)\,\frac{x^n}{n!} = \dfrac{2}{3-e^{2x}}
\]
\end{theorem}

\begin{proof}
Let $\Lambda=\{1231,1321\}$. 
We will find a differential equation satisfied by $A(x)=\sum_{n\ge0} \cc_n(\Lambda)\,\frac{x^n}{n!}$.

The coefficient of $\frac{x^n}{n!}$ in the derivative $A'(x)$ counts permutations $\pi\in\cC_{n+1}(\Lambda)$.
As in Lemma~\ref{lem:decomposition1}, such permutations can be written as $\pi=\alpha1\beta1\gamma$, where $\beta$ has no repeated entries, and $\Set(\alpha)\cap\Set(\gamma)=\emptyset$. Additionally, $\beta$ is either empty or has length $1$, since two distinct values in $\beta$ would create an occurrence of $1231$ or $1321$.

If $\beta=\emptyw$, then the standardized words $\st(\alpha)$ and $\st(\gamma)$ are arbitrary $\Lambda$-avoiding nonnesting permutations.

If $\beta=b$ for some $b\in\{2,3,\dots,n+1\}$, and the other copy of $b$ appears in $\alpha$, then $\st(\alpha b)$ and $\st(\gamma)$ are again arbitrary $\Lambda$-avoiding nonnesting permutations (with the caveat that $\st(\alpha b)$ is nonempty). If the other copy of $b$ appears in $\gamma$, the same is true for $\st(\alpha)$ and $\st(b\gamma)$.

It follows that $\pi\in\cC_{n+1}(\Lambda)$ equals one of the following: 
\begin{enumerate}[(1)]
    \item $\alpha11\gamma$, where $\st(\alpha)\in\cC_k(\Lambda)$ and $\st(\gamma)\in\cC_{n-k}(\Lambda)$ for some $0\le k\le n$,
    \item $\alpha1b1\gamma$, where $\st(\alpha b)\in\cC_k(\Lambda)$ and $\st(\gamma)\in\cC_{n-k}(\Lambda)$ for some $1\le k\le n$,
    \item $\alpha1b1\gamma$, where $\st(b\gamma)\in\cC_k(\Lambda)$ and $\st(\alpha)\in\cC_{n-k}(\Lambda)$ for some  $1\le k\le n$.
\end{enumerate}

Summing over $n\ge0$, case (1) contributes $A(x)^2$ to the exponential generating function, since $\Set(\alpha)$ and $\Set(\gamma)$ form an arbitrary ordered partition of $\{2,3,\dots,n+1\}$ into two nonempty sets, see e.g.\ \cite[Prop.~5.1.1]{EC2}. Each of the cases (2) and (3) contributes $(A(x)-1)A(x)$ because one of the blocks is nonempty. This gives the differential equation
$$A'(x)=A(x)^2+2(A(x)-1)A(x)=3A(x)^2-2A(x),$$
with initial condition $A(0)=1$. Solving this equation, we deduce that $A(x)=\frac{2}{3-e^{2x}}.$
\end{proof}

\begin{theorem} \label{thm:1231,1321,2113}
The ordinary generating function for nonnesting permutations that avoid $\{1231,1321,\allowbreak2113\}$ is
\[
\sum_{n\ge0} \cc_n(1231,1321,2113)\,x^n = \dfrac{1+2x-\sqrt{1-8x+4x^2}}{6x}.
\]
\end{theorem}

\begin{proof}
Let $\Lambda=\{1231,1321,2113\}$. 
Decomposing permutations $\pi\in\cC_{n+1}(\Lambda)$ as in the proof of Theorem~\ref{thm:1231,1321}, the additional condition of avoiding $2113$ requires $\alpha>\gamma$. Thus, $\Set(\gamma)=\{2,3,\dots,k+1\}$ and $\Set(\alpha)=\{k+2,k+3,\dots,n+1\}$ for some $0\le k\le n$, which makes the use of ordinary generating functions suitable in this case. 

Letting $B(x)=\sum_{n\ge0} \cc_n(\Lambda)\,x^n$, the same three cases as in the proof of Theorem~\ref{thm:1231,1321}, plus the empty permutation, give the equation
$$B(x)=1+xB(x)^2+2x(B(x)-1)B(x).$$
Solving for $B(x)$, we obtain the stated expression for the generating function.
\end{proof}

\section{Further research}\label{sec:further}

In a preprint version of this article, we proposed the open problem of finding a formula for the number of noncrossing permutations avoiding a single pattern in $\cS_3$. This problem has recently been solved for the pattern $132$ in~\cite{archer_laudone}. The question remains open for the pattern~$123$, as it does in the noncrossing case studied in~\cite{archer}.

\begin{problem}
    Find an expression for $\cc_n(123)$.
\end{problem}

The values of $\cc_n(123)$ for $1\le n\le 8$ are $1,4,17,82,406,2070,10729,56394,\dots$. This sequence does not appear in the Online Encyclopedia of Integer Sequences~\cite{OEIS} at the time of writing this paper.

For nonnesting permutations avoiding sets of patterns of length~4, we have presented some results in Section~\ref{sec:length4}, but there are many other sets to be considered. 
In Table~\ref{tab:conjecture} we list some cases that seem to give interesting enumeration sequences. All the conjectures have been checked for $n$ up to $8$.

\begin{table}[htb]
\centering
\begin{tabular}{|c|c|c|}
        \hline 
        $\Lambda$ & Conjecture for $\cc_n(\Lambda)$ & OEIS code\\
        \hline\hline
         $\{1322\}$ & $\displaystyle\frac{1}{n} \sum_{k=0}^{n-1} \binom{3n}{k} \binom{2n-k-2}{n-1}$  & A007297 \\
        \hline
        $\{1132,2213\}$ & \multirow{2}{*}{OGF: \ $\dfrac{(1-x)^2-\sqrt{(1-x)^4-4x(1-x)^2}}{2x}$} & \multirow{2}{*}{A006319}\\
        \cline{1-1}
        $\{1233,1322\}$ & & \\
        \hline
        $\{1132,3312\}$ & $3^n-3\cdot2^{n-1}+1$ & A168583\\
        \hline
        $\{1231,1312,2231,3221\}$    & OGF: \ $\dfrac{1-3x+2x^2}{(1-3x)(1-x-x^2)}$  & A099159\\
        \hline
        \end{tabular}
\caption{Some conjectures on the enumeration of nonnesting permutations avoiding other patterns. }
\label{tab:conjecture}
\end{table}

We also note that for some of the sets of patterns in Table~\ref{tab:length4-3ormore} we arrived at the same enumeration formulas, such as $\Cat_{n+1}-1$, using different proof methods. It would be interesting to find direct bijections explaining these Wilf equivalences.
In the same vein, we wonder if there is a simple bijective proof of Theorem~\ref{thm:132,213}, namely, a bijection between $\cC_n(132,213)$ and pairs of Fibonacci objects of the same size.

Tables~\ref{tab:length3}, \ref{tab:length4-atmost2} and~\ref{tab:length4-3ormore} show that
some of the enumeration sequences of pattern-avoiding nonnesting permutations are constant, others are polynomials, others grow exponentially, and others grow factorially. It would be interesting to understand the possible asymptotic behaviors of these sequences, the nature of their generating functions, and how these are determined by the properties of the avoided patterns.

Finally, in~\cite{elizalde_nonnesting}, it is shown that the polynomial enumerating all nonnesting permutations with respect to the number of descents has an unexpectedly simple factorization, and that its coefficients are palindromic. This suggests the study of the distribution of the number of descents on pattern-avoiding nonnesting permutations. Combining \cite[Thm.~2.6]{elizalde_nonnesting}
with our Lemma~\ref{lem:ijjk}, we obtain similar factorizations for the descent polynomials of nonnesting permutations avoiding patterns of the form $ijjk$. Specifically, denoting the number of descents by $\des(\alpha_1\dots \alpha_k)=|\{i:\alpha_i>\alpha_{i+1}\}|$ and the Narayana polynomials by $N_n(t)=\sum_{d=0}^{n-1}\frac{1}{n}\binom{n}{d}\binom{n}{d+1} t^d$, we have the following 
refinement of Lemma~\ref{lem:ijjk_enumeration}.

\begin{theorem}
    Let $\Sigma\subseteq\cS_3$, and let $\Lambda=\{\sigma_1\sigma_2\sigma_2\sigma_3:\sigma\in\Sigma\}$.
Then, for any $n\ge1$, 
$$\sum_{\pi\in\cC_n(\Lambda)} t^{\des(\pi)}=N_n(t) \sum_{\hat\pi\in\cS_n(\Sigma)} t^{\des(\hat\pi)}.$$
\end{theorem}

It follows, for example, that $\sum_{\pi\in\cC_n(1332)} t^{\des(\pi)}=N_n(t)^2$. In particular, the distribution of the number of descents over $\cC_n(1332)$ is symmetric. It is not hard to show that this distribution is also symmetric over $\cC_n(132,231)$, $\cC_n(121)$ and $\cC_n(112)$. It would be interesting to determine which sets of patterns have this property.

\acknowledgements
\label{sec:ack}
We thank Kate Kucharczuk for useful observations in the early stages of this work, Robert Dougherty-Bliss and Ryan Maguire for their help with some computations, and the referees for helpful comments. 

\nocite{*}
\bibliographystyle{abbrvnat}
\bibliography{nonnesting_references}

@article {archer,
    AUTHOR = {Archer, Kassie and Gregory, Adam and Pennington, Bryan and
              Slayden, Stephanie},
     TITLE = {Pattern restricted quasi-{S}tirling permutations},
   JOURNAL = {Australas. J. Combin.},
  FJOURNAL = {The Australasian Journal of Combinatorics},
    VOLUME = {74},
      YEAR = {2019},
     PAGES = {389--407},
   MRCLASS = {05A05 (05A15)},
  MRNUMBER = {3969738},
MRREVIEWER = {Yan\ Zhuang},
}

@misc{archer_laudone,
      title={Pattern avoidance in non-crossing and non-nesting permutations}, 
      author={Kassie Archer and Robert P. Laudone},
	   year={2025},
      eprint={2502.13309},
      archivePrefix={arXiv},
      primaryClass={math.CO}, 
      note = {arXiv:2502.13309},
}

@article {athanasiadis,
    AUTHOR = {Athanasiadis, Christos A.},
     TITLE = {On noncrossing and nonnesting partitions for classical
              reflection groups},
   JOURNAL = {Electron. J. Combin.},
  FJOURNAL = {Electronic Journal of Combinatorics},
    VOLUME = {5},
      YEAR = {1998},
     PAGES = {Research Paper 42, 16},
   MRCLASS = {05E15 (05A18 20F55)},
  MRNUMBER = {1644234},
MRREVIEWER = {Kimmo\ Eriksson},
}

@article {elizalde_quasiStirling,
    AUTHOR = {Elizalde, Sergi},
     TITLE = {Descents on quasi-{S}tirling permutations},
   JOURNAL = {J. Combin. Theory Ser. A},
  FJOURNAL = {Journal of Combinatorial Theory. Series A},
    VOLUME = {180},
      YEAR = {2021},
     PAGES = {Paper No. 105429, 35},
   MRCLASS = {05A05 (05A15 05A19)},
  MRNUMBER = {4211010},
MRREVIEWER = {Istv\'{a}n\ Mez\H{o}},
}

@article{elizalde_nonnesting,
    AUTHOR = {Elizalde, Sergi},
     TITLE = {Descents on nonnesting multipermutations},
   JOURNAL = {European J. Combin.},
  FJOURNAL = {European Journal of Combinatorics},
    VOLUME = {121},
      YEAR = {2024},
     PAGES = {Paper No. 103846, 23},
   MRCLASS = {05A05 (05A15)},
  MRNUMBER = {4778822},
}

@article {yan_statistics,
    AUTHOR = {Yan, Sherry H. F. and Yang, Lihong and Huang, Yunwei and Zhu,
              Xue},
     TITLE = {Statistics on quasi-{S}tirling permutations of multisets},
   JOURNAL = {J. Algebraic Combin.},
  FJOURNAL = {Journal of Algebraic Combinatorics. An International Journal},
    VOLUME = {55},
      YEAR = {2022},
    NUMBER = {4},
     PAGES = {1265--1277},
   MRCLASS = {05A15 (05A05)},
  MRNUMBER = {4423530},
}

@article {gessel_stanley,
    AUTHOR = {Gessel, Ira and Stanley, Richard P.},
     TITLE = {Stirling polynomials},
   JOURNAL = {J. Combinatorial Theory Ser. A},
  FJOURNAL = {Journal of Combinatorial Theory. Series A},
    VOLUME = {24},
      YEAR = {1978},
    NUMBER = {1},
     PAGES = {24--33},
   MRCLASS = {05A15},
  MRNUMBER = {462961},
MRREVIEWER = {Stephen\ Tanny},
}

@article {hoang,
    AUTHOR = {Hoang, Hung P. and M\"{u}tze, Torsten},
     TITLE = {Combinatorial generation via permutation languages. {II}.
              {L}attice congruences},
   JOURNAL = {Israel J. Math.},
  FJOURNAL = {Israel Journal of Mathematics},
    VOLUME = {244},
      YEAR = {2021},
    NUMBER = {1},
     PAGES = {359--417},
   MRCLASS = {68W40 (68Q45)},
  MRNUMBER = {4344032},
MRREVIEWER = {Mikhail\ V.\ Volkov},
}

@misc {OEIS,
    Author = {{OEIS Foundation Inc.}},
    Note = {Published electronically at \url{http://oeis.org}},
    Title = {The {O}n-{L}ine {E}ncyclopedia of {I}nteger {S}equences},
    Year = {2023}
    }

@article {bernardi,
    AUTHOR = {Bernardi, Olivier},
     TITLE = {Deformations of the braid arrangement and trees},
   JOURNAL = {Adv. Math.},
  FJOURNAL = {Advances in Mathematics},
    VOLUME = {335},
      YEAR = {2018},
     PAGES = {466--518},
   MRCLASS = {52C35 (05A15 05A19)},
  MRNUMBER = {3836672},
MRREVIEWER = {Piotr\ Pokora},
}

@article {simion_schmidt,
    AUTHOR = {Simion, Rodica and Schmidt, Frank W.},
     TITLE = {Restricted permutations},
   JOURNAL = {European J. Combin.},
  FJOURNAL = {European Journal of Combinatorics},
    VOLUME = {6},
      YEAR = {1985},
    NUMBER = {4},
     PAGES = {383--406},
   MRCLASS = {05A05 (05A15)},
  MRNUMBER = {829358},
}

@book{macmahon,
    author = {MacMahon, Percy A.},
    title = {Combinatory Analysis},
    publisher = {Cambridge University Press},
    year = {1915}
}

@book {knuth,
    AUTHOR = {Knuth, Donald E.},
     TITLE = {The art of computer programming. {V}ol. 1},
   EDITION = {Third},
      NOTE = {Fundamental algorithms},
 PUBLISHER = {Addison-Wesley, Reading, MA},
      YEAR = {1997},
     PAGES = {xx+650},
      ISBN = {0-201-89683-4},
   MRCLASS = {68-02},
  MRNUMBER = {3077152},
}

@article {ErdosSzekeres,
    AUTHOR = {Erd\H{o}s, P. and Szekeres, G.},
     TITLE = {A combinatorial problem in geometry},
   JOURNAL = {Compositio Math.},
  FJOURNAL = {Compositio Mathematica},
    VOLUME = {2},
      YEAR = {1935},
     PAGES = {463--470},
   MRCLASS = {99-04},
  MRNUMBER = {1556929},
}

@book {EC2,
    AUTHOR = {Stanley, Richard P.},
     TITLE = {Enumerative combinatorics. {V}ol. 2},
    SERIES = {Cambridge Studies in Advanced Mathematics},
    VOLUME = {62},
 PUBLISHER = {Cambridge University Press, Cambridge},
      YEAR = {1999},
     PAGES = {xii+581},
      ISBN = {0-521-56069-1; 0-521-78987-7},
   MRCLASS = {05A15 (05-02 05E05 05E10 68R05)},
  MRNUMBER = {1676282},
MRREVIEWER = {Ira\ Gessel},
}

@book {Stanley_Catalan,
    AUTHOR = {Stanley, Richard P.},
     TITLE = {Catalan numbers},
 PUBLISHER = {Cambridge University Press, New York},
      YEAR = {2015},
     PAGES = {viii+215},
      ISBN = {978-1-107-42774-7; 978-1-107-07509-2},
   MRCLASS = {05-01 (01A05 11B75 11B83)},
  MRNUMBER = {3467982},
MRREVIEWER = {David\ Callan},
}

@article {Elizalde_canon,
    AUTHOR = {Elizalde, Sergi},
     TITLE = {Canon permutations and generalized descents of standard
              {Y}oung tableaux},
   JOURNAL = {Enumer. Comb. Appl.},
  FJOURNAL = {Enumerative Combinatorics and Applications},
    VOLUME = {5},
      YEAR = {2025},
    NUMBER = {2},
     PAGES = {Paper No. S2R13, 10},
   MRCLASS = {05A05 (05A15 05A19)},
  MRNUMBER = {4867903},
}
\label{sec:biblio}

\end{document}